\documentclass[11pt,a4paper, parskip=half]{scrartcl} %option parskip=half
\usepackage[utf8]{inputenc}
\usepackage[english]{babel}
\usepackage[T1]{fontenc}
\usepackage{lmodern}
\usepackage[left=3cm,right=3cm,top=3cm,bottom=3cm]{geometry}
\usepackage[runin]{abstract}
    \abslabeldelim{.}

\usepackage{todonotes}

\usepackage{amsmath, amsfonts, amssymb, amsthm, amstext}
\usepackage{mathtools}
\usepackage{mathrsfs}
\usepackage{mathdots}
\usepackage{cases}

\font\mfett=cmmib10 at11pt
 at9pt
\newcommand\mycom[2]{\genfrac{}{}{0pt}{}{#1}{#2}}

\usepackage{enumitem}

\usepackage{hyperref} %, nameref}
\usepackage[capitalise]{cleveref}
	
\newcounter{thm}
\numberwithin{thm}{section}
\numberwithin{equation}{section}

	\newtheoremstyle{myplain}		% name of the style to be used
			{}			% measure of space to leave above the theorem. E.g.: 3pt
			{}			% measure of space to leave below the theorem. E.g.: 3pt
			{\itshape}				% name of font to use in the body of the theorem
			{}				% measure of space to indent
			{\sffamily\bfseries}				% name of head font
			{.}		% punctuation between head and body
			{ }				% space after theorem head; " " = normal interword space
			{\thmname{#1}\thmnumber{ #2}\textnormal{\textsf{\thmnote{ (#3)}}}}			% Manually specify head
    \newtheoremstyle{mybreak}
            {}{}{}{}{\sffamily\bfseries}{.}{\newline}
            {\thmname{#1}\thmnumber{ #2}\textnormal{\textsf{\thmnote{ (#3)}}}}
	\newtheoremstyle{mydef}
			{}{}{}{}{\sffamily\bfseries}{.}{ }
			{\thmname{#1}\thmnumber{ #2}}
	\newtheoremstyle{myrem}
			{}{}{}{}{\sffamily\itshape}{.}{ }
			{\thmname{#1}\thmnumber{ #2}}

\theoremstyle{myplain}
	\newtheorem{theorem}[thm]{Theorem}

    \newtheorem{corollary}[thm]{Corollary}
\theoremstyle{mybreak}
	\newtheorem{algorithm}[thm]{Algorithm}
\theoremstyle{mydef}
	
	\newtheorem{remark}[thm]{Remark}
\theoremstyle{myrem}

	\def\gamra{\hbox{\mfett\char013}} %% fettes \gamma

\newcommand{\argmax}{\mathop{\mathrm{argmax}}}
\newcommand{\argmin}{\mathop{\mathrm{argmin}}}
\newcommand{\atan}{\mathop{\mathrm{atan2}}}

\allowdisplaybreaks
\def\sumprime_#1^#2{
    \setbox0=\hbox{$\scriptstyle{#1}$}
    \setbox1=\hbox{$\scriptstyle{#2}$}
    \setbox2=\hbox{$\displaystyle{\sum}$}
    \setbox4=\hbox{${}^\prime\mathsurround=0pt$}
    \dimen0=.5\wd0 \advance\dimen0 by-.5\wd2
    \ifdim\dimen0>0pt
        \ifdim\dimen0>\wd4 \kern\wd4
        \else\kern\dimen0
        \ifdim\dimen1>\wd4 \kern\wd4
        \else\kern\dimen1
    \fi\fi\fi
\mathop{{\sum}^\prime}_{\kern-\wd4 #1}^{\kern-\wd4 #2}
}

\title{\LARGE Exact Reconstruction of Sparse Non-Harmonic Signals from Fourier Coefficients}
\author{Markus Petz\footnote{Institute for Numerical and Applied Mathematics,n G\"ottingen University, Lotzestr.\ 16-18, 37083 G\"ottingen, Germany, \{m.petz,plonka,n.derevianko\}@math.uni-goettingen.de} \quad Gerlind Plonka$^{*}$\footnote{Corresponding author} \quad  Nadiia Derevianko$^{*}$}
\date{}

\begin{document}
	\let\oldproofname=\proofname
	\renewcommand{\proofname}{\itshape\sffamily{\oldproofname}}

\maketitle

\begin{abstract}
In this paper, we derive a new reconstruction method for real non-harmonic Fourier sums, i.e., real signals which can be represented as sparse exponential sums of the form  $f(t) = \sum_{j=1}^{K} \gamma_{j} \, \cos(2\pi a_{j} t + b_{j})$,  where the frequency parameters $a_{j} \in {\mathbb R}$ (or $a_{j} \in {\mathrm i} {\mathbb R}$) are pairwise different. Our  method is based on the recently proposed stable iterative rational approximation algorithm in \cite{NST18}. For signal reconstruction we use a set of classical Fourier coefficients of $f$ with regard to a fixed interval $(0, P)$ with $P>0$.  Even though all terms of $f$ may be non-$P$-periodic, our reconstruction method requires 
at most $2K+2$ Fourier coefficients $c_{n}(f)$ to recover all parameters of $f$.
We show that in the case of exact data, the proposed iterative algorithm terminates after at most $K+1$ steps.  
The algorithm can also detect the number $K$ of terms of $f$, if $K$ is  a priori unknown  and  $L>2K+2$ Fourier coefficients are available.
Therefore our method provides a new stable alternative to the known numerical approaches for the recovery of exponential sums that are based on Prony's method.\\
\textbf{Keywords:}  sparse exponential sums, non-harmonic Fourier sums, reconstruction of sparse non-periodic signals, rational approximation, AAA algorithm, barycentric representation,
Fourier coefficients.\\
\textbf{AMS classification:}
41A20, 42A16, 42C15, 65D15, 94A12.
\end{abstract}

%%%%%%%%%%%%%%%%%%%%%%%%%%%%%%%%%%%%%%%%%%%%%%%%%%%%%%%%%%%%%%%%%%%%%%%%%%%%%%%%%%%%%%%%%%%%%%%%%%%%%%%%
\section{Introduction}
Classical Fourier analysis methods provide for any real square integrable signal $f(t)$   a Fourier  series representation on a given interval  $(0,P)$, $P>0$, of the form 
\begin{equation}\label{fo} f(t) =  \sum_{n= - \infty}^{\infty} c_{n}(f) \, {\mathrm e}^{2\pi {\mathrm i}n t/P}
= \frac{\alpha_{0}\, \cos(\beta_{0}) }{2} +\sum_{n=1}^{\infty} \alpha_{n} \, \cos\left(\frac{2\pi n t}{P}  - \beta_{n}\right)
\end{equation}  
with Fourier coefficients 
$$ c_{n}(f) = \frac{1}{P}  \int\limits_{0}^{P} f(t) \, {\mathrm e}^{2\pi {\mathrm i} nt/P} \, {\mathrm d}t , \qquad n \in {\mathbb Z},$$
and $\alpha_{n}=2\, |c_{n}(f)|$, $\beta_{n} = \atan(\mathrm{Im} \, c_{n}(f), \mathrm{Re} \, c_{n}(f))$ for  $n \ge 0$, where $\atan$  denotes the modified  inverse tangent, 
see e.g.\ \cite{PPST18}, Chapter 1.
If $f$ is $P$-periodic and differentiable, then its Fourier series (\ref{fo})  converges uniformly to $f$.
 However, if $f$ is smooth but not $P$-periodic, then the $P$-periodization  of $f$ ``forced'' by the Fourier series representation in (\ref{fo}) usually 
 leads to a discontinuity at the interval boundaries $t=0$ and  $t=P$, respectively, and thus to a slow decay of the Fourier coefficients.

In applications, it frequently happens that a signal is only given on an interval of length $P$, where it appears to be non-periodic,
even if $f$  may be periodic with a certain period $P_{1}$ which is not of the form $P/n$ for some positive integer $n$.
Considering for example the signal 
\begin{equation}\label{bsp} f(t) = \cos(2\pi \sqrt{2} t) + \cos(2\pi \sqrt{3} t), 
\end{equation}
which contains only two different frequency parameters, 
we observe that this signal is non-periodic with regard to any interval $(0, P)$. For $P=1$, the  corresponding Fourier series is given by
$$  \resizebox{\textwidth}{!}{$
f(t) = \!\! \sum\limits_{n=-\infty}^{\infty} \!\!\left( \frac{\sin(\sqrt{2}\pi) ( \sqrt{2} \cos(\sqrt{2} \pi) + {\mathrm  i} \, n \sin(\sqrt{2} \pi) )}{\pi(2-n^{2})}  + \frac{\sin(\sqrt{3}\pi) ( \sqrt{3} \cos(\sqrt{3} \pi) + {\mathrm  i} \, n \sin(\sqrt{3} \pi) )}{\pi(3-n^{2})}  \right)
\, {\mathrm e}^{2\pi {\mathrm i} nt}. $}$$
Thus, the question occurs, how to reconstruct a non-harmonic Fourier sum, i.e., how to compute the much more informative representation (\ref{bsp}) directly from suitable measurements of $f$.

\textbf{Contents of this paper.} 
The goal of this paper is to reconstruct  non-harmonic Fourier sums $f$  
of the form 
\begin{equation}\label{1.1}
f(t) = \sum_{j=1}^K \, \gamma_j \, \cos(2 \pi a_j  t + b_j) = \sum_{j=1}^{K} f_{j}(t)
\end{equation}
from a finite number of its classical Fourier coefficients $c_{n}(f)$ corresponding to a Fourier series of $f$ on $(0, P)$. 
Here, we assume that 
$K \in {\mathbb N}$, $\gamma_{j} \in (0, \infty)$, and  $(a_{j}, \, b_{j}) \in (0, \infty) \times  [0, \, 2\pi)$, 
and that the frequency parameters $a_{j}$ are pairwise distinct.
As we will show in the sequel, the restrictions made for $K, \, \gamma_{j}, \, a_{j}$, and $b_{j}$ will ensure uniqueness of the presentation (\ref{1.1}).
Note that $f$ in (\ref{1.1}) admits real (nonnegative) frequency parameters $a_{j}$ and therefore essentially generalizes usual trigonometric polynomials. 
The example in (\ref{bsp}) with frequencies $\sqrt{2}$ and $\sqrt{3}$ is covered by our model (\ref{1.1}) taking $K=2$, $\gamma_{1}= \gamma_{2} = 1$, $a_{1}= \sqrt{2}$, $a_{2} =\sqrt{3}$,  and $b_{1}=b_{2}=0$.
Observe that the function $f$ in (\ref{1.1}) is only $P$-periodic for some $P >0$ if all parameters $a_{j}$ can be written in the form $a_{j} = c \, q_{j}$ with a positive constant $c \in {\mathbb R}$ and a rational number $q_{j} \ge 0$, i.e., only in this case, there exists a real number $P >0$ such that $f(t+P) = f(t)$ for all $t \in {\mathbb R}$. 

\medskip

In Section 2 we show that our model (\ref{1.1}) is well-defined, i.e., that  all parameters  $K, a_{j}, \, b_{j}, \gamma_{j}$, $j=1,\ldots , K$, are (with the given restrictions) uniquely determined for a non-harmonic Fourier sum $f$. 
If all terms  $f_{j}$ of $f$ are non-$P$-periodic, i.e., if all frequency parameters $a_{j} > 0$ in  (\ref{1.1}) satisfy that $a_{j} \not\in \frac{1}{P} {\mathbb N}$, then it is shown that the modified Fourier coefficients $\tilde{c}_{n}(f) := \textrm{Re}\, c_{n}(f) + \frac{1}{n} \textrm{Im} \, c_{n}(f)$, $n >0$, have a special structure of the  form $r(n^{2})$, where $r(z)$ is a rational function of type $(K-1,K)$. Conversely, $r(z)$ already provides all information to find the parameters 
determining $f$ in (\ref{1.1}).

Section 3 is devoted to the new reconstruction method.
Using a modification of the recently proposed AAA algorithm for iterative rational approximation, see \cite{NST18}, we compute $r(z)$ from a set of (at least) $2K+1$ classical Fourier coefficients of $f$. Then a partial fraction decomposition together with a non-linear bijective transform 
provides the wanted parameters in (\ref{1.1}).
Numerical stability of the rational approximation algorithm is ensured using a barycentric representation of the numerator and the denominator polynomial of $r(z)$. Compared to other rational interpolation algorithms, a further important advantage of the employed modified AAA algorithm is that we  do not need a priori knowledge on the number $K$ of terms in (\ref{1.1}) but can determine $K$ in the iteration process, supposed that $L \ge 2K+1$ Fourier coefficients are available.

We show in Section 4, that a signal $f$ with $K$ non-$P$-periodic terms $f_{j}$ as in (\ref{1.1})  can already be determined from $2K$ Fourier coefficients $c_{n}(f)$  with $k \in \Gamma \subset {\mathbb N}$. Moreover, our method based on the AAA algorithm always provides the wanted rational function $r(z)$ after $K$ iteration steps (and using $K+1$ modified Fourier coefficients).

In Section 5, our new  reconstruction method is generalized to the case that $f$ in (\ref{1.1}) also contains $P$-periodic terms $f_{j}$ with frequencies 
$a_{j} \in \frac{1}{P} {\mathbb N}$. It turns out that there is no a priori information needed about possibly occurring $P$-periodic terms of $f$.  In this case, we first compute the rational function $r(z)$ that determines the non-$P$-periodic part of $f$, where we again employ the  modified AAA algorithm from Section 3.  Afterwards, the $P$-periodic terms of $f$ can be found in a post-processing step, if all $c_{n_{j}}(f)$ with $n_{j} = Pa_{j}$ are contained in the given set of Fourier coefficients.
In particular, we show that $f$ can be always completely recovered from $2K+2$ Fourier coefficients. 

In Section 6, we generalize  the model in (\ref{1.1}). Beside supposing $(a_{j}, \, b_{j}) \in (0, \infty) \times [0, 2\pi)$, we can also admit parameters
$(a_{j}, \, b_{j}) \in {\mathrm i} (0, \infty) \times {\mathrm i} {\mathbb R}$.  By $\cos({\mathrm i}x) = \cosh(x)$,  this leads  to terms of the form $\gamma_{j}\cosh((-{\mathrm i})(2\pi a_{j} + b_{j})$ in (\ref{1.1}). The considered generalization still admits the same rational structure of Fourier coefficients  and can therefore be treated similarly as (\ref{1.1}) with the proposed reconstruction method.

Finally we provide some numerical experiments. The Matlab implementation of our reconstruction algorithm is provided at the Software section of our homepage\newline
 \url{http://na.math.uni-goettingen.de}.

\medskip

\textbf{Related literature.}
Our model (\ref{1.1}) can be viewed as a sparse expansion into exponentials with $2K$ terms via Euler's identity, i.e.,
$$ f(t) = \sum_{j=1}^{K} \left( \frac{\gamma_{j}}{2} {\mathrm e}^{{\mathrm i} b_{j}} \right)  \, {\mathrm e}^{2\pi {\mathrm i} a_{j} t} + \left( \frac{\gamma_{j}}{2} {\mathrm e}^{-{\mathrm i} b_{j}} \right)  \, {\mathrm e}^{-2\pi {\mathrm i} a_{j} t}. $$
Exponential sums have been extensively studied within the last years, based on Prony's method and its relatives, see  e.g.\ \cite{BDB10,BM05,CL18,FMP12,PPT11,PSK18,PT14,PT10,PT2013,RK89,VMB02,ZP18}.
To reconstruct an exponential sum via Prony's method, one usually employs equidistant function values $f(t_{0} + h\ell)$, $\ell = 0, \ldots , L$, and the number of given values should be at least $2M$, where $M$ is the number of terms in the exponential sum. In our case, the number of exponential terms is $M=2K$, but the symmetry properties can be exploited such that the samples $f(h\ell)$, $\ell=0, \ldots, 2K-1$, are theoretically sufficient for the recovery of $f$ in the noise-free case, see e.g.\ \cite{PSK18}.
However, Prony's method involves Hankel or Toeplitz matrices with possibly high condition numbers, and therefore requires a very careful numerical treatment.
Our new method for reconstruction of signals of the form (\ref{1.1}) in this paper is based on rational approximation and can be seen as a good alternative to the Prony reconstruction approach.

Another way to look at the model  (\ref{1.1}) is to view it as a special case of  a signal decomposition into so-called intrinsic mode functions (IMFs)  in adaptive data analysis, see \cite{Huang}.
Empirical mode decomposition (EMD) is based on a model that decomposes the  signal $f$ into $K$ IMFs,
\begin{equation}\label{emd} f(t) = \sum_{j=1}^{K}  \gamma_{j}(t) \, \cos( \phi_{j}(t)) \end{equation}
with  nonnegative envelope functions $\gamma_{j}(t)$ and  so-called instantaneous phase functions $\phi_{j}(t)$, see e.g.\ \cite{Huang}.
As already pointed out in \cite{DLW11}, despite certain restrictions, as e.g.\ that $\gamma_{j}(t)$ and $\phi_{j}(t)$ are smooth with  $\gamma_{j}(t) \ge 0$ and $\phi_{j}'(t) \ge 0$ for $t \in {\mathbb R}$, a representation of the form (\ref{emd}) is far from being unique.
For example, the function $f(t)$ in (\ref{bsp}) has the form  (\ref{emd})
with $K=2$, constant functions $\gamma_{1}(t)$, $\gamma_{2}(t)$, and with 
$\phi_{1}(t)= 2\pi \, \sqrt{2}t$ and $\phi_{2}(t)= 2\pi \, \sqrt{3}t$.
However, $f(t)$ in (\ref{bsp})  can also be written as a single IMF  in $\Big[- \Big(\frac{\sqrt{3} + \sqrt{2}}{4}\Big), \frac{\sqrt{3} + \sqrt{2}}{4}\Big]$, 
$$ f(t) =\left(2 + 2 \, \cos(2\pi \, (\sqrt{3} - \sqrt{2})t \right)^{1/2} \, \cos\left(2\pi \, \frac{\sqrt{2} + \sqrt{3}}{2} \, t\right).
$$
The non-uniqueness of the model (\ref{emd}) often prevents a simple interpretation of the obtained decomposition.

Compared to (\ref{emd}), the main advantages of the non-harmonic Fourier sum  (\ref{1.1}) are  that  
the representation of $f$  is unique, 
and, that the model (\ref{1.1}) has a direct physical interpretation, similarly  to classical Fourier sums.

There are also other approaches to represent signals by adaptive generalized Fourier sums using the so-called Takenaka-Malmquist  basis, an adaptive orthonormal basis, see \cite{QW11,PP19}.
While the greedy algorithm in \cite{QW11} only slightly improves the signal approximation compared to classical Fourier sums, it has been shown in \cite{PP19}, that strong decays of adaptive Fourier expansions can be achieved, if the sequence of classical Fourier coefficients of a signal can be well approximated using a short exponential sum. Our approach in the current paper is somehow vice versa, the (modified) Fourier coefficients of $f$ are  represented by rational functions in order to  reconstruct the special exponential sum $f$.

While we focus on signal reconstruction in the current paper, there remains the question of (almost) optimal signal approximation by non-harmonic Fourier sums, which we will study in the future.
Obviously, each square integrable signal in $(0,P)$ can be arbitrarily well approximated bei a non-harmonic Fourier sum (\ref{1.1}) for $K \to \infty$, since it is a direct generalization of classical Fourier sums. However, approximation rates for signals in certain function classes are not completely known so far. 
Research on non-harmonic Fourier series particularly focussed on functional analytic questions, see e.g.\
\cite{Young80}. In particular, it has been shown that  $\{{\mathrm e}^{2\pi {\mathrm i}a_{j} \cdot} \}_{j \in {\mathbb Z}}$ forms a Riesz basis in $L^{2}([0,1])$ for a given increasing sequence $\{ a_{j}\}_{j \in {\mathbb Z}}$ if and only if $|a_{j} - j| < 1/4$ for $j \in {\mathbb Z}$, see \cite{Kad64}, while completeness of this function system is ensured for $|a_{j}| \le |j| + 1/4$, where $a_{j}$ can even be complex, see \cite{Lev40}.
Note that for finite non-harmonic Fourier sums as in (\ref{1.1}) we do not need any further assumption on the  distribution of (pairwise distinct) frequencies %$a_{j} \ge 0$ 
to ensure the uniqueness of the presentation.

%%%%%%%%%%%%%%%%%%%%%%%%%%%%%%%%%%%%%%%%%%%%%%%%%%%%%%%%%%%%%%%%%%%%%%%%%%%%%%%%%%%%%%%%%%%%%%%%%%%%%%%%
\section{Non-Harmonic  Signals}

We consider signals $f$ of the form (\ref{1.1}) 
with $K \in {\mathbb N}$, $\gamma_{j} \in (0, \infty)$, 
and $(a_{j}, \, b_{j}) \in (0, \infty) \times  [0, \, 2\pi)$,  and
we assume that the parameters $a_{j}$, $j=1, \ldots , K$, are pairwise distinct.

\subsection{Unique Representation of  Non-Harmonic Signals}

We will show that the model in (\ref{1.1}) is  well-defined, since the occurring parameters $K, \, \gamma_{j}, \, a_{j}$ and $b_{j}$, $j=1, \ldots , K$, are uniquely determined for a function $f$ given on an interval  with positive length.
More precisely, we can show the following:

\begin{theorem} \label{theo1}
Let $f$ be given as in $(\ref{1.1})$ with $K \in {\mathbb N}$, $\gamma_j \in (0, \infty)$, and $(a_j , \, b_j) \in (0, \infty) \times [0, 2\pi)$, %or  $(a_j , \, b_j)= (0, \, 0)$,
where  $a_{1} < a_{2} < \ldots < a_{K}$. Further, let 
$$ g(t) = \sum_{j=1}^{M} \delta_{j} \, \cos(2 \pi c_{j}  t + d_{j}) $$
with $M \in {\mathbb N}$, $\delta_j \in (0, \infty)$,  and 
$(c_j , \, d_j) \in (0, \infty) \times [0, 2\pi)$, 
where $c_{1 } < c_{2} < \ldots < c_{M}$.
If $f(t) = g(t)$ for all $t$ on an interval $T \subset {\mathbb R}$ of positive length,  then we have 
$K=M$ and $\gamma_{j}=\delta_{j}$, $a_{j}=c_{j}$, $b_{j}= d_{j}$ for $j=1, \ldots , K$.
\end{theorem}

\begin{proof}
1. We consider $h(t) = f(t) - g(t)$. Then $h(t) = 0$  for all $t \in T$, and $h(t)$ has the structure
\begin{equation}\label{h} h(t)= \sum_{j=1}^{K+M} \mu_{j} \, \cos(2\pi x_{j}t + y_{j}) 
\end{equation}
with
\begin{eqnarray*}
 \mu_{j} &:=& \gamma_{j}, \qquad \qquad x_{j} := a_{j}, \qquad \quad y_{j}:= b_{j}, \quad j=1, \ldots , K, \\
 \mu_{K+j} &:=& -\delta_{j}, \qquad x_{K+j} := c_{j}, \qquad y_{K+j}:= d_{j}, \quad j=1, \ldots , M.
 \end{eqnarray*}
By assumption, the number $L$ of distinct frequency parameters $x_{j}$ in the representation (\ref{h}) of $h$ satisfies $L \ge \max \{K, \, M\}$, and we can rewrite $h(t)$ in the form
\begin{equation}\label{hhh} h(t) = \sum_{\ell=1}^{L} \alpha_{\ell}\, \cos(2\pi \tilde{x}_{\ell} t) - \beta_{\ell} \, \sin(2\pi \tilde{x}_{\ell}t), 
\end{equation}
where $\tilde{x}_{\ell} \in \{ x_{j}: \, j=1, \ldots,  K+M\}$ are now pairwise distinct.
If $\tilde{x}_{\ell}$ occurs only once in the set $\{ x_{j}: \, j=1, \ldots,  K+M \} $, 
say $\tilde{x}_{\ell} = x_{j}$, then 
$$  \alpha_{\ell} =  \mu_{j} \, \cos(y_{j})  \qquad \beta_{\ell} = \mu_{j} \, \sin(y_{j}) .$$
If $\tilde{x}_{\ell}$ occurs twice in the set $\{ x_{j}: \, j=1, \ldots , K+M \} $, say
$\tilde{x}_{\ell}= x_{j_{1}}= x_{K+j_{2}}$, with $j_{1} \in \{1, \ldots , K\}$ and $j_{2} \in \{1, \ldots , M\}$, then 
\begin{equation}\label{ab} \alpha_{\ell} = \mu_{j_{1}} \cos(y_{j_{1}}) + \mu_{K+j_{2}} \cos(y_{K+j_{2}}), \qquad 
\beta_{\ell} =  \mu_{j_{1}} \sin(y_{j_{1}}) + \mu_{K+j_{2}} \sin(y_{K+j_{2}}). \end{equation}

2. 
We show that the $2L$ functions $\{ \cos(2\pi \tilde{x}_{\ell} t), \, \sin(2\pi \tilde{x}_{\ell}t): \, \ell=1, \ldots , L\}$ occurring in 
(\ref{hhh}), are linearly independent on $T$. 
  By Euler's identity there is  an invertible transform from this function set to $\{ {\mathrm e}^{2\pi {\mathrm i} \tilde{x}_{\ell}t}, \, {\mathrm e}^{-2\pi {\mathrm i} \tilde{x}_{\ell}t}: \, \ell=1, \ldots , L\}$, i.e.,
  $h(t)$ in  (\ref{hhh}) can also be written as 
\begin{equation}\label{hh} h(t) = \frac{1}{2} \sum_{\ell=1}^{2L} \xi_{\ell} \, {\mathrm e}^{2\pi {\mathrm i} \tilde{x}_{\ell}t},  
\end{equation}
with $\tilde{x}_{L+\ell} :=-\tilde{x}_{\ell}$ and
with ${\xi}_{\ell} = \alpha_{\ell} + {\mathrm i} \beta_{\ell}$ as well as ${\xi}_{L+\ell} = \alpha_{\ell} - {\mathrm i} \beta_{\ell}$, $\ell=1, \ldots, L$.
We obtain for the  Wronskian  of the function system $\{ {\mathrm e}^{2\pi {\mathrm i} \tilde{x}_{\ell} t}: \, \ell=1, \ldots , 2L\}$ that 
\begin{eqnarray*}
 \mathrm{det} \! \left( \frac{{\mathrm d}^{k}}{{\mathrm d} t^{k}} {\mathrm e}^{2\pi {\mathrm i} \tilde{x}_{\ell} t} \right)_{k=0,\ell=1}^{2L-1,2L} \!\!\!\!
&=& \!\!   \mathrm{det}  \left( \begin{array}{cccc}
1 & 1 & \ldots & 1 \\
2\pi {\mathrm i} \tilde{x}_{1} & 2\pi {\mathrm i} \tilde{x}_{2} & \ldots & 2\pi {\mathrm i} \tilde{x}_{2L} \\
\vdots & & & \vdots \\
\!\!(2\pi {\mathrm i} \tilde{x}_{1})^{2L-1}\! &\!\! (2\pi {\mathrm i} \tilde{x}_{2})^{2L-1} \!&\!\ldots & \!\!(2\pi {\mathrm i} \tilde{x}_{2L})^{2L-1}\! \end{array} \right) \, \prod\limits_{\ell=1}^{2L} {\mathrm e}^{2\pi {\mathrm i} \tilde{x}_{\ell}t}  \\[1ex]
&=& \det \left( \tilde{x}_{\ell}^{k} \right)_{k=0,\ell=1}^{2L-1,2L} \, (2\pi {\mathrm i})^{L(2L-1)}  \neq 0 
\end{eqnarray*}
for all $t \in T \subset {\mathbb R}$, since the Vandermonde matrix $\left( \tilde{x}_{\ell}^{k} \right)_{k=0,\ell=1}^{2L-1,2L}$ is invertible for 
 pairwise distinct $\tilde{x}_{\ell}$, $\ell=1, \ldots, 2L$, and we have 
 $\prod\limits_{\ell=1}^{2L} {\mathrm e}^{2\pi {\mathrm i} \tilde{x}_{\ell}t} = \prod\limits_{\ell=1}^{L} {\mathrm e}^{2\pi {\mathrm i} \tilde{x}_{\ell}t} {\mathrm e}^{-2\pi {\mathrm i} \tilde{x}_{\ell}t} = 1$.
 Thus, linear independence of ${\mathrm e}^{2\pi {\mathrm i} \tilde{x}_{\ell}t}$, $\ell=1, \ldots, 2L$, and hence  of $\cos(2\pi \tilde{x}_{\ell} t), \, \sin(2\pi \tilde{x}_{\ell}t)$, $\ell=1, \ldots , L$, follows.

3. Now, if $\tilde{x}_{\ell}$ would occur only once in the set $\{ x_{j}: \, j=1, \dots , K+M\}$, 
say  $\tilde{x}_{\ell} = x_{j}$, then $\alpha_{\ell} = \beta_{\ell} =0$ implies
$\mu_{j} \, \cos(y_{j})=0$ and  $\mu_{j} \, \sin(y_{j}) =0$, 
and thus $\mu_{j}=0$ contradicting the assumption.
Therefore, $\tilde{x}_{\ell}$ always occurs twice, and it follows already that $K=M=L$. 
Let 
$\tilde{x}_{\ell}= x_{j_{1}}= x_{K+j_{2}}$, with $j_{1}, \, j_{2} \in \{1, \ldots , K\}$. Thus we find $a_{j_{1}} = c_{j_{2}}$. Further,  $\alpha_{\ell} = \beta_{\ell} =0$ implies by (\ref{ab})  that 
\begin{equation}\label{det1}  \textrm{det} \left(\begin{array}{cc}
\cos y_{j_{1}} & \cos y_{K+j_{2}}\\
\sin y_{j_{1}} & \sin y_{K+j_{2}} \end{array} \right) = -\sin(y_{j_{1}}-y_{K+j_{2}})=0.
\end{equation}
We use the assumption  $y_{j_{1}}=b_{j_{1}} \in [0, 2\pi)$ and $y_{K+j_{2}} = d_{j_{2}} \in [0, 2\pi)$, and conclude from $y_{j_{1}}-y_{K+j_{2}} \in \{-\pi, \, 0,\, \pi\}$ that 
either $b_{j_{1}} = d_{j_{2}}$ or $b_{j_{1}} = d_{j_{2}} + \pi \, \mathrm{mod} \, 2\pi$. However, in the second case it would follow that  $\cos d_{j_{2}} = - \cos b_{j_{1}}$ and $\sin d_{j_{2}} = - \sin b_{j_{2}}$, and thus by (\ref{ab})
$$ \left(\begin{array}{cc}
\cos b_{j_{1}} & \cos d_{j_{2}}\\
\sin b_{j_{1}} & \sin d_{j_{2}} \end{array} \right) \left( \begin{array}{c}
\mu_{j_{1}} \\ \mu_{K+j_{2}} \end{array} \right)
= \left(\begin{array}{cc}
\cos b_{j_{1}} & -\cos b_{j_{1}}\\
\sin b_{j_{1}} & -\sin b_{j_{1}} \end{array} \right) \left( \begin{array}{c}
\gamma_{j_{1}} \\ -\delta_{j_{2}} \end{array} \right) = {\mathbf 0} $$
contradicting  the assumption $\mu_{j_{1}}=\gamma_{j_{1}} >0$ and $\mu_{K+j_{2}} = -\delta_{j_{2}} <0$. Hence, $b_{j_{1}} = d_{j_{2}}$ and $\gamma_{j_{1}} = \delta_{j_{2}}$.
Since these conclusions are valid for each $\tilde{x}_{\ell}$, the assertion of the theorem follows.
\end{proof}

\begin{remark}
1. Theorem \ref{theo1} also shows that a function $f$ of the form (\ref{1.1}) with the given restrictions on the parameters $\gamma_{j}, \, a_{j}, \, b_{j}$ cannot vanish on any interval $T \subset {\mathbb R}$  with positive length.

2. The linear independence of the system $\{ {\mathrm e}^{2\pi {\mathrm i} \tilde{x}_{\ell} t}, \, {\mathrm e}^{-2\pi {\mathrm i} \tilde{x}_{\ell} t}: \, \ell=1, \ldots, L\}$ in the proof of Theorem \ref{theo1} also follows from the fact that an exponential sum of the form $h(t)$ in (\ref{hh}) can appear as a general solution of a linear difference equation of order $2L$ with constant coefficients, see e.g.\ \cite{Berg86}.

3. Observe that the function model can simply be extended by adding a constant component 
$f_{0}(t) = \pm \gamma_{0} = \gamma_{0} \, \cos(2\pi a_{0} t + b_{0})$ with $\gamma_{0} >0$, 
$a_{0}=0$ and either $b_{0}=0$ for $f_{0}(t)>0$ or  $b_{0}= \pi$ for $f_{0}(t)<0$. 
This extended model also satisfies  the assertion of Theorem \ref{theo1}. If we admit beside $(a_{j}, b_{j}) \in (0, \infty) \times [0, 2\pi)$ also $(a_{j}, b_{j}) = (0,0)$  and $(a_{j}, b_{j}) = (0,\pi)$, the proof of Theorem \ref{theo1} can be suitably modified. 
\end{remark}

\subsection{Classical Fourier Coefficients of  Non-Harmonic Signals}

Now we study  the Fourier coefficients of structured functions of the form $\phi(t) = \gamma \, \cos( 2\pi a  t + b)$ with $\gamma  \in (0, \infty)$, and 
 $(a, \, b) \in (0, \infty) \times [0, 2\pi) $ 
within the interval $[0, P)$ for given $P>0$. 

\begin{theorem}\label{theo2}
Let $\phi(t) = \gamma \, \cos( 2\pi a  t + b)$ with $\gamma  \in (0, \infty)$,  $(a, \, b) \in (0, \infty) \times [0, 2\pi) $,   
and let $P>0$.
Then $\phi$ possesses in $[0, P)$ the Fourier series 
$$ \phi(t) = \sum_{n \in {\mathbb Z}} c_n(\phi) \, {\mathrm e}^{2\pi {\mathrm i} nt/P}$$
with Fourier coefficients
$ c_n(\phi) = \frac{1}{P} \int_0^P \phi(t) \, {\mathrm e}^{-2\pi i n t/P} \, {\mathrm d} t,$ for $n \in {\mathbb Z},$
where
\begin{eqnarray} \label{Rc}
\mathrm{Re} \, c_n(\phi) &=&  \frac{P \gamma a}{\pi (a^2 P^2-n^2)} \sin(a \pi P) \, \cos(a \pi P + b),\\
\label{Ic}
\mathrm{Im} \, c_n(\phi) &=&  \frac{ \gamma n}{\pi(a^2 P^2 - n^2)} \sin(a \pi P) \, \sin(a \pi P + b)
\end{eqnarray}
for $n \in {\mathbb  N}_{0}$, and $c_{-n}(\phi) = \overline{c_{n}(\phi)}$.
If $a \in \frac{1}{P} {\mathbb N}$,  then the Fourier coefficients of $\phi$ simplify to 
$$ c_{n}(\phi)  = \left\{\begin{array}{ll} 
\frac{\gamma}{2} (\cos (b) - {\mathrm i} \sin(b)) & \mathrm{for} \quad n=  a P > 0,\\
\quad 0 & \mathrm{for} \quad n \in {\mathbb N} \setminus \{a P\}.
\end{array} \right.
$$
Pointwise convergence of the Fourier series for $\phi(t)$ is given for all $t \in (0,P)$.
\end{theorem}

\begin{proof}
Since $\gamma \, \cos( 2\pi a  t + b)$ is a differentiable function, its restriction onto the interval $(0,P)$ is also differentiable. Thus the Fourier expansion of $\phi$ converges pointwise for all $t \in (0, P)$, see \cite{PPST18}, Chapter 1. 
For the real part of $c_n(\phi)$ we obtain with $\cos x \cos y = \frac{1}{2}( \cos(x+y) + \cos(x-y))$
\begin{eqnarray*}
\mathrm{Re} \, c_n(\phi)  &=& \frac{1}{P} \int_0^P \gamma \, \cos( a 2\pi t + b) \, \cos\Big(\frac{2\pi n t}{P}\Big) {\mathrm d} t \\
&=& \frac{\gamma}{2P} \int_0^P \left( \cos\left(2\pi t\Big(a+\frac{n}{P}\Big)+b\right) + \cos\left(2\pi t\Big(a-\frac{n}{P}\Big)+b\right) \right) {\mathrm d}t.
\end{eqnarray*}
Assuming that $a \neq  \frac{n}{P}$ it follows with $\sin x - \sin y = 2 \sin (\frac{x-y}{2} ) \cos (\frac{x+y}{2} ) $
\begin{eqnarray*}
\mathrm{Re} \, c_n(\phi)  &=&  \frac{\gamma}{2P} \left( \frac{\sin\Big(2\pi t(a+\frac{n}{P})+b\Big)}{2\pi (a+\frac{n}{P})}\Big|_0^P + \frac{\sin\Big(2\pi t(a-\frac{n}{P})+b\Big)}{2\pi (a-\frac{n}{P})}\Big|_0^P \right) \\
&=& \frac{\gamma}{2P \pi} \sin(\pi a P)\cos(\pi a P + b) \left( \frac{1}{a+\frac{n}{P}} + \frac{1}{a-\frac{n}{P}} \right)\\
&=& \frac{\gamma a P}{\pi (P^2 a^2 - n^2)} \sin(\pi a P)\cos(\pi a P + b).
\end{eqnarray*}
For $a = \frac{n}{P}>0$, the function $\phi$ is $P$-periodic, and we simply find
$$
\mathrm{Re} \, c_n(\phi)  = \frac{\gamma}{2P} \int_0^P \cos \left(2\pi t\Big( \frac{2n}{P}\Big)+b\right) + \cos(b) \, {\mathrm d}t 
= \frac{\gamma}{2} \cos(b).
$$
This is also achieved from (\ref{Rc}) by taking the limit with the rule of L' Hospital, 
$$\lim_{a \to  \frac{n}{P}} \mathrm{Re} \, c_n(\phi) = \lim_{a \to  \frac{n}{P}} \frac{P \gamma a}{\pi (a^2 P^2-n^2)} \sin(a \pi P) \, \cos(a \pi P + b) = \frac{\gamma}{2} \cos(b).$$
The formula (\ref{Ic}) for the imaginary part of $c_n(\phi)$ can be derived analogously.
\end{proof}

\begin{remark} If $\phi(t)= \gamma \, \cos(2\pi at + b)$ is constant, i.e., if $\gamma>0$ and 
either $(a,b) = (0,0)$ or $(a,b) = (0,\pi)$, then we obtain the Fourier coefficients $c_{0}(\phi)= \gamma \cos(b)$, and $c_{n}(\phi) = 0$ for $n \in {\mathbb Z} \setminus \{ 0\}$. 
\end{remark}

\subsection{Representation of Fourier  Coefficients by Rational Functions}
\label{sec:rat}

We consider now functions $f= \sum_{j=1}^{K} f_{j}$ of the form (\ref{1.1})
with $f_{j}(t) = \gamma_{j} \cos(2\pi a_{j} t + b_{j})$,  $\gamma_j \in (0, \infty)$ and  $(a_j, \, b_{j})  \in (0, \infty) \times [0, 2\pi)$,
where the $a_{j}$ are assumed to be pairwise distinct.
As shown in Theorem  \ref{theo2}, we have for $n \in {\mathbb Z}$ and
 $a_j \not\in  \frac{1}{P} {\mathbb N}$ that 
\begin{equation}\label{fc} \mathrm{Re} \, c_n(f_j) = \frac{A_j}{n^2-C_j}, \qquad  \mathrm{Im} \, c_n(f_j) = \frac{B_j n}{n^{2}- C_j},
\end{equation}
where 
\begin{eqnarray} \label{cj}
C_j &:=& a_j^2\, P^2,\\
\label{aj}
A_j &:=& -\frac{P \gamma_j \, a_j}{\pi} \, \sin(a_j \pi P) \, \cos( a_j \pi P + b_j), \\
\label{bj}
 B_j &:=&  - \frac{\gamma_j}{\pi} \, \sin(a_j \pi P) \, \sin( a_j \pi P + b_j). 
\end{eqnarray}
Note that $C_{j}$ is real and positive for real values $a_{j}$.
 We show that $A_{j}, \, B_{j},\,  C_{j}$, $j=1, \ldots , K$, completely determine all Fourier coefficients $c_{n}(f)$ and thus $f$.

\begin{theorem} \label{theo3}
Let $f$ be given as in $(\ref{1.1})$ with $\gamma_{j} \in (0, \infty)$ and  $(a_j, \, b_{j})  \in (0, \infty) \times [0, 2\pi)$. 
Further let $a_j \not\in  \frac{1}{P} {\mathbb N}$  be pairwise distinct.
Then, there is a bijection between the parameters $\gamma_j, \, a_j, \, b_j$, $j=1, \ldots , K$, determining $f(t)$ 
and the parameters $A_j,\,  B_j, \, C_j$, $j=1, \ldots , K$, in $(\ref{cj})-(\ref{bj})$.  
We have $C_{j} >0$, 
$$
 a_j = \frac{1}{P} \sqrt{C_{j}}, \qquad 
 \gamma_j = \frac{\pi}{\sqrt{C_j} \, |\sin(\sqrt{C_{j}} \pi)|} \sqrt{A_j^2+C_{j} B_j^2}.
 $$
  For $B_{j} \neq 0$, 
$$
b_{j} \, \mathrm{mod} \, \pi = \left( \mathrm{arccot} \left( \frac{A_{j}}{\sqrt{C_{j}} B_{j}} \right) - \pi \sqrt{C_{j}} \right) \mathrm{mod} \, \pi,
$$
and 
$ b_{j} \in [0, \, \pi)$ for  $\mathrm{sign}\Big(-A_{j} + B_{j} \sqrt{C_{j}} \cot(\sqrt{C_{j}} \pi) \Big) >0$ and $ b_{j} \in [\pi,\,  2\pi)$ otherwise. 
Here, $\mathrm{arccot}$ denotes the inverse cotangens  that maps onto $[-\pi/2, \pi/2]$. 
For $B_{j}=0$, choose $b_{j}$ from $\{ -\pi \sqrt{C_{j}} \, \mathrm{mod} \, \pi, -\pi \sqrt{C_{j}} \, \mathrm{mod} \, \pi + \pi \}$ such that $(\ref{aj})$ is satisfied.
\end{theorem}

\begin{proof}
We can assume  that $C_{j} > 0$. Then (\ref{cj}) implies
 that $a_j= \frac{\sqrt{C_j}}{P} > 0$. 
Further, taking  the weighted sum $A_{j}^{2} + a_{j}^{2} P^{2}  B_{j}^{2}$ using (\ref{aj}) and (\ref{bj}) we obtain
$$ \gamma_j^{2} =\frac{\pi^{2}}{a_{j}^{2} P^{2} \, (\sin(a_{j} P \pi))^{2}} ({A_j^2+ a_{j}^{2} P^{2} B_j^2}) = \frac{\pi^{2}}{C_{j} \, (\sin(\sqrt{C_{j}}\pi))^{2}} ({A_j^2+ C_{j} B_j^2}).$$
Since $\gamma_{j} >0$, we can determine $\gamma_{j}$ uniquely. 
Inserting the found representations for $a_j$ and $\gamma_j$  into (\ref{aj}) and (\ref{bj}), we conclude  for $B_{j} \neq 0$
$$ \cot(a_{j} \pi P + b_{j}) = \frac{ A_{j}}{P a_{j} \, B_{j}} = \frac{ A_{j}}{\sqrt{C_{j}} \, B_{j}} $$
as well as 
$$ -A_{j} + B_{j} \, a_{j} \, P \, \cot(a_{j} \pi P) =  \frac{P \gamma_{j} \, a_{j}}{\pi} \sin(b_{j}), $$
and thus $\mathrm{sign}(-A_{j} + B_{j} \,\sqrt{C_{j}} \, \cot(\sqrt{C_{j}} \pi)) =  \mathrm{sign} (\sin(b_{j}))$. If $B_{j}=0$ then $\sin(a_{j} \pi P + b_{j}) =0$  and thus $b_{j} \in \{ -\pi \sqrt{C_{j}} \, \mathrm{mod} \, \pi, -\pi \sqrt{C_{j}} \, \mathrm{mod} \, \pi + \pi \}$.
\end{proof}

\begin{remark}
Since we had assumed that $a_{j} \not\in \frac{1}{P} {\mathbb N}$ and in particular $a_{j} \neq 0$, it follows that $C_{j}  \neq 0$. As seen from Theorem \ref{theo2}, we always have $C_{j} > 0$ for the considered model.
In  Section \ref{sec:general}, we will generalize the model to treat also the case  $C_{j} <0$ which leads to generalized expansions  involving also cosine hyperpolic terms.
\end{remark}

 For $f$ in (\ref{1.1}) (with $a_{j} \not\in  \frac{1}{P} {\mathbb N}$) we observe with (\ref{fc}) that
\begin{equation}\label{cnrat}  c_n(f) = \sum_{j=1}^K \left( \frac{A_j}{n^2-C_{j}} +  {\mathrm i} \frac{B_j n}{n^2-C_{j}} \right).
\end{equation}
In particular, the real part $\textrm{Re} \, c_n(f)$ and the imaginary part $\textrm{Im} \, c_n(f)$ can for all $n \in {\mathbb Z}$ be written as
$ \textrm{Re} \, c_n(f) = \frac{p^R_{K-1}(n^2)}{q_K(n^2)}$ and  $\textrm{Im} \, c_n(f) =\frac{n \, p^I_{K-1}(n^2)}{q_K(n^2)}$,
where 
\begin{equation}\label{qform} q_K(z):= \prod_{j=1}^K (z-C_{j}) 
\end{equation}
is a monic polynomial of degree $K$, and  where
\begin{equation}\label{pform}
p_{K-1}(z):= p^R_{K-1}(z) + {\mathrm i}\, p^I_{K-1}(z) = \sum_{j=1}^K (A_j + {\mathrm i} \,  B_j) \prod_{\mycom{\ell=1}{\ell \neq j}}^K ( z - C_{\ell}) 
\end{equation}
is a (complex) algebraic polynomial of degree (at most) $K-1$. In other words, for $n \neq 0$,
\begin{equation}\label{tilde}
 \tilde{c}_{n}(f) := \textrm{Re} \, c_n(f) + \frac{{\mathrm i}}{n} \textrm{Im} \, c_n(f) = r_{K}(n^2) = \frac{p_{K-1}(n^2)}{q_{K}(n^2)}, 
\end{equation}
i.e., the \textit{modified Fourier coefficient} $\tilde{c}_{n}(f)$ can be represented by a rational function  of type $(K-1,K)$, evaluated at $z=n^{2}$, where $q_K(z)$ in (\ref{qform})  and $p_{K-1}(z)$
in (\ref{pform}) are coprime.

 %%%%%%%%%%%%%%%%%%%%%%%%%%%%%%%%%%%%%%%%%%%%%%%%%%%%%%%%%%%%%%%%%%%%%%%%%%%%%%%%%%%%%%%%%%%%%%%%%%%%%%%%
\section{Modified AAA Algorithm for Sparse Signal Representation}
\label{sec:AAA}
%%%%%%%%%%%%%%%%%%%%%%%%%%%%%%%%%%%%%%%%%%%%%%%%%%%%%%%%%%%%%%%%%%%%%%%%%%%%%%%%%%%%%%%%%%%%%%%%%%%%%%%%

We want to exploit the special structure of the Fourier coefficients of functions $f$ which are built by function atoms of the form $f_{j}(t) = \gamma_{j} \, \cos(2 \pi a_{j}t + b_{j})$ in order to study the following  problem:

How  can we reconstruct  a function $f$ of the form (\ref{1.1})  from a given set of its Fourier coefficients in a stable and efficient way?
We assume here that only the structure of $f$ is known, i.e., we need to recover the number $K$ of terms in (\ref{1.1}) as well as the parameters  $\gamma_{j}, \, a_{j}, b_{j}$ for $j=1, \ldots , K$.

For the reconstruction process we need to keep in mind that the rational representation of Fourier coefficients in (\ref{cnrat}), or in (\ref{tilde}) respectively, 
is only valid for the non-$P$-periodic terms $f_{j}$ of $f$, i.e., for $a_{j} \not\in \frac{1}{P} {\mathbb N}_{0}$. If $f$ contains components 
$f_{j}(t)= \gamma_{j}  \cos(2 \pi a_{j} t + b_{j})$ with $a_{j} P = n_{j} \in {\mathbb N}_{0}$, then, as shown in Theorem \ref{theo3}, 
these components will provide only one non-zero Fourier coefficient $c_{n_{j}}(f)$ (with nonnegative index), which destroys the rational 
function structure (\ref{tilde}) at $z=n_{j}^{2}$.

Therefore, our approach consists of two parts. In the first step, we will reconstruct the non-$P$-periodic part of $f$, and in a second step, 
we will determine possible $P$-periodic terms of $f$ that can be  obtained  from the set of Fourier coefficients.
\medskip

To reconstruct the non-$P$-periodic part of $f$ in (\ref{1.1}) we will extensively use the structure of the Fourier coefficients $c_{n}(f)$ found in (\ref{tilde}) and employ a modification of the recently proposed AAA algorithm in \cite{NST18}. 
Differently from other rational interpolation algorithms, the modified AAA algorithm provides essentially higher numerical stability and enables us to determine also the order of the rational approximant which is at the same time the number $K$ of terms in (\ref{1.1}). 
The AAA algorithm can be seen as a method for rational approximation, where certain values of a given function are interpolated, while other given values are approximated using a least squares approach.
With this algorithm, we  will determine the rational function in (\ref{tilde})  by  interpolation or approximation of the given  modified Fourier coefficients $\tilde{c}_{n}(f)$.
The algorithm works iteratively, where at each iteration step the degree of the polynomials determining the rational function grows by $1$, and a next interpolation value is chosen at the point, where the error of the rational approximation found so far is maximal.
The algorithm terminates if either the error at all given points considered for approximation is less than a given bound or if a certain fixed degree of the rational function is reached.
If the rational function is found, we can extract the parameters $A_j, \, B_j, \, C_j$ and then finally obtain the wanted representation with parameters $\gamma_j, \, a_j, \, b_j$ of the non-$P$-periodic part of $f$  from Theorem \ref{theo3}.

Stability of the AAA algorithm is ensured by a barycentric representation, as already considered in \cite{SW86,FH07} and exploited also in  \cite{FNTB18}, to compute rational minimax approximations.
In \cite{NST18}, the AAA-algorithm is presented for rational functions $r(z) = p(z)/q(z)$, where the polynomials $p$ and $q$ have the same degree. 
Therefore, we need to modify the approach for our purpose, similarly as proposed in \cite{FNTB18}.
As side effect of the linearization procedure within the algorithm is that unattainable interpolation points lead to vanishing weight components, see \cite{SW86}.
This behavior of the algorithm enables us to determine possible 
$P$-periodic terms of $f$ in a postprocessing step, where we need to inspect all Fourier coefficients that cannot be well approximated by the found rational function. 

If $f$ does not contain $P$-periodic terms and the given Fourier coefficients of $f$ are exact, then we will be able to determine $f$ uniquely from $2K+1$ Fourier coefficients. This will be shown in Section \ref{sec:exact}. Otherwise, we will need $2K+2$ Fourier coefficients to recover $f$, where all $P$-periodic terms are simply determined in a postprocessing step, see Section \ref{sec:periodic}.

\subsection{Rational Interpolation using Barycentric Representation}
\label{sec:start}

Let us assume now that we are given a set of Fourier coefficients $c_n(f)$, $n \in \Gamma \subset {\mathbb N}$
of the function $f$ of the form (\ref{1.1}) with $L := \# \Gamma \ge 2K+1$.
We assume first that all terms of $f$ are non-$P$-periodic, such that we obtain the rational structure of $\tilde{c}_{n}= \tilde{c}_{n}(f)$ as given in (\ref{tilde}).
We want to find a rational function $r_{K}(z) = p_{K-1}(z)/q_{K}(z)$ of type $(K-1, K)$ such that the interpolation conditions
$$ r_{K}(n^{2}) = \tilde{c}_{n}, \qquad n \in \Gamma $$
are satisfied. Assuming that the given modified Fourier coefficients $\tilde{c}_{n}$ of $f$ in (\ref{1.1}) are exact,  we will show in Section \ref{sec:exact} that  $r_{K}(z)$ will be the wanted rational function in (\ref{tilde}) that determines $f$.

As in \cite{NST18,FH07,Klein}, we will use the  barycentric representation of $r_{K}(z)  = \tilde{p}_{K}(z)/\tilde{q}_{K}(z)$ with 
\begin{equation}\label{rat0} \tilde{p}_{K}(z) := \sum_{j=1}^{K+1} \frac{w_j \, \tilde{c}_{n_j}}{z-n_j^2}, \qquad  \tilde{q}_{K}(z) := \sum_{j=1}^{K+1} \frac{w_j}{z-n_j^2}, 
\end{equation}
where $w_j$, $j=1, \ldots , K+1$, are nonzero weights, and where $n_{j} \in \Gamma$ for $j=1, \ldots, K+1$.
Here, $n_{j}^{2}$, $j=1, \ldots , K+1$, cannot  occur as poles of $r_{K}(z)$,  since the poles  $C_{j}$ in (\ref{qform}) satisfy $C_{j}= a_{j}^{2} P^{2} \not\in {\mathbb N}$ by assumption. 
It can be simply observed that $\tilde{p}_{K}(z)/\tilde{q}_{K}(z)$ is indeed a rational function of type $(K,K)$. In order to ensure that $r_{K}(z)$ is of the wanted type $(K-1,K)$, we require the additional condition $\sum_{j=1}^{K+1} w_{j} \, \tilde{c}_{n_{j}} =0$. 

The representation (\ref{rat0}) already incorporates the interpolation conditions $r_{K}(n_{k}^{2}) = \tilde{c}_{n_{k}}(f)$, for $k=1, \ldots, K+1$, since 
we have  for $w_{k} \neq 0$
$$
\lim_{z \to n_{k}^{2}} \frac{\tilde{p}_{K}(z)}{\tilde{q}_{K}(z)} = \lim_{z \to n_{j}^{2}} \frac{\sum\limits_{j=1}^{K+1} w_{j} \, \tilde{c}_{n_{j}} \, \prod\limits_{\mycom{\ell=1}{\ell\neq j}}^{K+1} (z- n_{j}^{2}) }{\sum\limits_{j=1}^{K+1} w_{j}  \, \prod\limits_{\mycom{\ell=1}{\ell\neq j}}^{K+1} (z- n_{j}^{2})} 
=\frac{w_{k} \, \tilde{c}_{n_{k}}}{w_{k}} = \tilde{c}_{n_{k}}. 
$$
Let $S_{K+1} :=\{n_{1}, \ldots , n_{K+1}\}$ be the index set, where  for nonzero weights $w_{k}$ the interpolation conditions are already satisfied.
To determine $r_{K}(z)$ using (\ref{rat0}) we still need to fix the normalized weight vector ${\mathbf w}=(w_{1}, \ldots , w_{K+1})^{T}$.
According to \cite{NST18}, this is done by solving a least squares problem in order to minimize the error 
$$ \sum_{n \in \Gamma \setminus S_{K+1}} |\tilde{c}_{n} \, \tilde{q}_{K}(n^{2}) - \tilde{p}_{K}(n^{2}) |^{2}, $$
where beside $\| {\mathbf w} \|_{2} =1$, in our case we will incorporate the side condition $\sum_{j=1}^{K+1} w_{j} \, \tilde{c}_{n_{j}} =0$ to ensure that $r_{K}(z)$ is of type $(K-1,K)$. 
 The algorithm is described in the next two sections and closely follows the approach in \cite{NST18} with the modification that we want to get a rational function of type $(K-1,K)$ instead of type $(K,K)$. 

\subsection{Initialization of the Modified AAA Algorithm}
\label{sec:ini}

We start by initializing the modified AAA-algorithm as follows.
First, we  choose the two given modified Fourier coefficients  $\tilde{c}_{n_1}$, $\tilde{c}_{n_2}$ with largest modulus (where $n_1 \neq n_2$ and $n_1, \, n_2 \in \Gamma$) for interpolation and compute a rational function $r_{1}(z)$ of type $(0,1)$ that interpolates $\tilde{c}_{n_1}$ at $z=n_{1}^{2}$ and $\tilde{c}_{n_2}$ at $z=n_{2}^{2}$.
We  determine the rational function
$$ r_1(z) = \frac{\tilde{c}_{n_1} \tilde{c}_{n_2} (n_2^2 - n_1^2)}{(\tilde{c}_{n_2} n_2^2- \tilde{c}_{n_1} n_1^2) + z(\tilde{c}_{n_1}- \tilde{c}_{n_2})}, $$
such that $r_1(n_1^2) = \tilde{c}_{n_1}$ and $r_1(n_2^2) = \tilde{c}_{n_2}$ holds. A barycentric form of  $r_1(z)$ as in (\ref{rat0}) is given by
$$ r_1(z) :=  \frac{\tilde{p}_1(z)}{\tilde{q}_1(z)} = \frac{\frac{w_1 \, \tilde{c}_{n_1}}{z-n_1^2} + \frac{w_2 \, \tilde{c}_{n_2}}{z-n_2^2}}{\frac{w_1}{z-n_1^2} + \frac{w_2}{z-n_2^2}}, $$
with the (complex) weights
\begin{equation}\label{ww} w_1 = \frac{-\tilde{c}_{n_2}}{\sqrt{|\tilde{c}_{n_1}|^2 + |\tilde{c}_{n_2}|^2}}, \qquad  w_2=\frac{\tilde{c}_{n_1}}{\sqrt{|\tilde{c}_{n_1}|^2 + |\tilde{c}_{n_2}|^2}} 
\end{equation}
satisfying $|w_1|^2+|w_2|^2=1$ and $w_1 \tilde{c}_{n_1} + w_2 \tilde{c}_{n_2} =0$. 
Obviously, $\tilde{p}_{1}(z)$ and $\tilde{q}_{1}(z)$  are themselves rational functions of type (at most) $(1,2)$.
The  condition $w_1 \tilde{c}_{n_1} + w_2 \tilde{c}_{n_2} =0$ 
ensures that the polynomial 
$$\tilde{p}_1(z) \, (z-n_1^2)(z-n_2^2) = w_{1} \tilde{c}_{n_{1}} (z-n_{2}^{2}) + w_{2} \tilde{c}_{n_{2}} (z-n_{1}^{2}) 
$$ 
is only constant (and not linear).

%The idea is now to iteratively enlarge the set of interpolation points (thereby enlarging the degrees of the numerator and denominator polynomial) to improve the rational approximant step by step.
In order to decide, which interpolation point should be taken at the next iteration step, 
we consider the error $|r_1(n^2) - \tilde{c}_n|$ for all $n \in \Gamma \setminus \{ n_{1}, n_{2}\}$. 
Following the lines of \cite{NST18}, we use
the notation  $S_2:= \{n_1, \, n_2 \} \subset \Gamma$ and $\Gamma_{2}: = \Gamma \setminus S_2$. Let the Cauchy matrix ${\mathbf C}_2$ be given by 
$ {\mathbf C}_2 := \left( \frac{1}{n^2-n_j^2}\right)_{n \in \Gamma_{2}, n_j \in S_{2}} $ with $2$ columns and $L-2$ rows. 
Then 
the vectors of function values $\left(\tilde{p}_1(n^2) \right)_{n \in \Gamma_{2}}$ and $\left(\tilde{q}_1(n^2) \right)_{n \in \Gamma_{2}}$ satisfy 
$$ \left(\tilde{p}_1(n^2) \right)_{n \in \Gamma_{2}} = {\mathbf C}_2 \left(\begin{array}{l} \!\!w_1 \, \tilde{c}_{n_1} \!\!\\ \!\!w_2 \, \tilde{c}_{n_2} \!\!\end{array} \right), \qquad
\left(\tilde{q}_1(n^2) \right)_{n \in \Gamma_{2}} = {\mathbf C}_{2} \left(\begin{array}{l} \!\!w_1 \!\! \\\!\! w_2 \!\! \end{array} \right), $$
and $\max\limits_{n \in \Gamma_{2}} |r_1(n^2) - \tilde{c}_n| = \max\limits_{n \in \Gamma_{2}} |\tilde{p}_{1}(n^2)/\tilde{q}_{1}(n^{2}) - \tilde{c}_n|$ can be easily computed.
We choose $\tilde{n} = \argmax\limits_{n \in \Gamma_{2}} |r_1(n^2) - \tilde{c}_n|$ as the next index for interpolation and set $S_3:= S_2 \cup \{ \tilde{n}\}$ and $\Gamma_{3}:= \Gamma_{2} \setminus \{ \tilde{n} \}$.

\subsection{General Iteration Step of the Modified AAA Algorithm}
\label{sec:gene}

At step $(J-1)>1$, assume that we have given the index set  $S_{J}:=\{n_1, \ldots , n_{J}\} \subset \Gamma$, where we want to interpolate, and let $\Gamma_{J} := \Gamma \setminus  S_{J}$. 
 Further, let 
$$\tilde{\mathbf c}_{S_{J}}:= \left( \tilde{c}_{n_j} \right)_{j=1}^{J} \in {\mathbb C}^J, \qquad  
\tilde{\mathbf c}_{\Gamma_{J}} := \left( \tilde{c}_{n} \right)_{n \in \Gamma_{J}} \in {\mathbb C}^{L-J},$$ 
be the given vectors of (modified) Fourier coefficients as in (\ref{tilde}), where we will take $\tilde{\mathbf c}_{S_{J}}$  for interpolation  and $\tilde{\mathbf c}_{\Gamma_{J}}$ for approximation.

We use a barycentric representation as in (\ref{rat0}) and start with the ansatz  $r_{J-1}(z)  = \tilde{p}_{J-1}(z)/\tilde{q}_{J-1}(z)$ with 
\begin{equation}\label{rat1} \tilde{p}_{J-1}(z) := \sum_{j=1}^J \frac{w_j \, \tilde{c}_{n_j}}{z-n_j^2}, \qquad  \tilde{q}_{J-1}(z) := \sum_{j=1}^J \frac{w_j}{z-n_j^2}, 
\end{equation}
 and weights $w_j$, $j=1, \ldots , J$.
Then $r_{J-1}(z)$ already satisfies 
the interpolation conditions $r_{J-1}(n_j^2) = \tilde{c}_{n_j}$ for  $w_{j} \neq 0$.
The vector of weights ${\mathbf w} := (w_1, \ldots , w_J)^T$ has still be chosen suitably
with the side  conditions
\begin{equation}\label{cond}
\| {\mathbf w} \|_{2}= 1 \qquad \textrm{and} \qquad {\mathbf w}^T  \,  \tilde{\mathbf c}_{S_{J}}= \sum_{j=1}^J w_j \tilde{c}_{n_j} = 0
\end{equation}
to ensure that $r_{J-1}(z)$ is of type $(J-2, J-1)$.
As in the original AAA algorithm, the remaining freedom to choose ${\mathbf w}$ is now used in order to approximate the (modified) Fourier coefficients $\tilde{c}_n$
by $r_{J-1}(n^2)$ for $n \in \Gamma_{J}$ applying a (linearized) least squares approach. Observing that 
$r_{J-1}(z) \tilde{q}_{J-1}(z) = \tilde{p}_{J-1}(z)$, we consider the minimization problem 
\begin{equation}\label{mini}
{\mathbf w}_J := \argmin_{\mycom{\| {\mathbf w}\|_2=1}{{\mathbf w}^T \tilde{\mathbf c}_{S_{J}=0}}} \sum_{n \in \Gamma_{J}} \left|\tilde{c}_{n} \, \tilde{q}_{J-1}(n^2)-\tilde{p}_{J-1}(n^2)\right|^2.
\end{equation}
Similarly to \cite{NST18}, we  define the matrices 
$$ 
{\mathbf A}_J:= \left( \frac{\tilde{c}_n - \tilde{c}_{n_j}}{n^2-n_j^2} \right)_{n \in \Gamma_{J}, n_j \in S_J} \in {\mathbb C}^{L-J \times J},\quad 
{\mathbf C}_J:= \left( \frac{1}{n^2-n_j^2} \right)_{n \in \Gamma_{J}, n_j \in S_J} \in {\mathbb R}^{L-J \times J}.$$
Then we can write
\begin{equation}\label{rater}
\sum_{n \in \Gamma_{J}} \left|\tilde{c}_{n} \, \tilde{q}_{J-1}(n^2)-\tilde{p}_{J-1}(n^2)\right|^2 = \sum_{n \in \Gamma_{J}} \left|       {\mathbf w}^T \, \left( 
\frac{\tilde{c}_n- \tilde{c}_{n_j}}{n^2-n_j^2} \right)_{j=1}^J \right|^2 
= \| {\mathbf A}_J {\mathbf w} \|_2^2, 
\end{equation}
such that the minimization problem in (\ref{mini}) takes the form
\begin{equation}\label{mini1}
{\mathbf w}_J := \argmin_{\mycom{\| {\mathbf w}\|_2=1}{{\mathbf w}^T \tilde{\mathbf c}_{S_{J}=0}}}  \| {\mathbf A}_J {\mathbf w} \|_2^2.
\end{equation}

To solve the minimization problem (\ref{mini1}) approximatively, we compute the right  (normalized) singular vectors ${\mathbf v}_1$ and ${\mathbf v}_2$ of ${\mathbf A}_J$ corresponding to the two smallest singular values  $\sigma_{1} \le \sigma_{2}$ of ${\mathbf A}_J$  and take a linear combination ${\mathbf w}_{J} = \mu_1 {\mathbf v}_1 + \mu_2 {\mathbf v}_2$ such that $\| {\mathbf w}_{J}\|_2=1$ and ${\mathbf w}_{J}^T \tilde{\mathbf c}_{S_J} =0$.
These conditions are satisfied for 
\begin{equation}\label{w} {\mathbf w}_J = 
\frac{1}{\sqrt{({\mathbf v}_1^T \tilde{\mathbf c}_{S_J})^2 + ({\mathbf v}_2^T \tilde{\mathbf c}_{S_J})^2}} \left( ({\mathbf v}_2^T \tilde{\mathbf c}_{S_J})   \, {\mathbf v}_1 - ({\mathbf v}_1^T \tilde{\mathbf c}_{S_J}) \, {\mathbf v}_2 \right). 
\end{equation}

\begin{remark}
Obviously, the right singular vector ${\mathbf v}_{1}$ already solves $\argmin\limits_{\|{\mathbf w}\|_{2} = 1} \| {\mathbf A}_{J}\,  {\mathbf w}\|_{2}^{2}$. 
The vector ${\mathbf w}_J$ in (\ref{w}) is a linear combination  of the two singular 
vectors corresponding to the  two smallest singular values $\sigma_{1} \le \sigma_{2}$ of ${\mathbf A}_{J}$, 
such that $\|{\mathbf A}_{J} {\mathbf w}_{J} \|_{2}^{2} \le \sigma_{2}^{2}$. 
The computed vector ${\mathbf w}_{J}$ in (\ref{w}) is optimal if $\sigma_{1} = \sigma_{2}$ or if ${\mathbf v}_{1}^{T} \, \tilde{\mathbf c}_{S_{J}} = 0$, i.e., if the singular vector ${\mathbf v}_{1}$ to the smallest singular value of ${\mathbf A}_{J}$ already satisfies the side condition (\ref{cond}), and ${\mathbf w}_{J} = {\mathbf v}_{1}$.
We will show in Section \ref{sec:exact} that for $J=K+1$ in case of exact data the matrix ${\mathbf A}_{K+1}$ always possesses a kernel vector that solves (\ref{w}).
\end{remark}

Having determined the weight vector ${\mathbf w}_J$, the rational function $r_{J-1}$ is completely fixed by (\ref{rat1}).
 Now, we consider the errors $| r_{J-1}(n^2) - \tilde{c}_n|$ for all $n \in \Gamma_{J}$, where we do not interpolate.
The algorithm terminates if $\max_{n \in \Gamma_{J}} | r_{J-1}(n^2) - \tilde{c}_n| < \epsilon$ for a predetermined bound $\epsilon$ or if $J-1$ reaches a predetermined maximal degree. Otherwise, we find the next index for interpolation as
$$ n_{J+1} := \argmax_{n \in \Gamma_{J}} | r_{J-1}(n^2) - \tilde{c}_n|. $$
The values $r_{J-1}(n^2) = \frac{\tilde{p}_{J-1}(n^2)}{\tilde{q}_{J-1}(n^2)}$ can be simply computed by the vectors
$$ (\tilde{p}_{J-1}(n^2))_{n \in \Gamma_{J}} = {\mathbf C}_J ({\mathbf w}_{J} \, .* \tilde{\mathbf c}_{S_J}), \quad (\tilde{q}_{J-1}(n^2))_{n \in \Gamma_{J}} = {\mathbf C}_J \, {\mathbf w}_{J} $$
as suggested in \cite{NST18}, where $.*$ indicates pointwise multiplication.

We summarize the modified AAA algorithm to compute the vectors ${\mathbf w}_{K+1}$ and the index vector ${\mathbf S}_{K+1} :=(n_{1}, \ldots , n_{K+1})^{T}$ determining the rational function $r_{K}(z) = \frac{p_{K-1}(z)}{q_{K}(z)} = \frac{\tilde{p}_{K}(z)}{\tilde{q}_{K}(z)}$ in (\ref{tilde}) resp.\ (\ref{rat1}) that interpolates the given modified Fourier coefficients $\tilde{c}_{n}$ if $n$ is a component of ${\mathbf S}_{K+1}$, and approximates $\tilde{c}_{n}$ otherwise. 

\begin{algorithm}[Rational approximation of modified Fourier coefficients by modified AAA]
\label{alg1}
\textbf{Input:} \\
$P>0$ period used for computing the Fourier coefficients\\
${\mathbf \Gamma}  \in {\mathbb N}^{L}$ vector of indices of given Fourier coefficients with  sufficiently large $L$ \\ 
$ {\mathbf c} \in {\mathbb C}^{L}$ vector of given Fourier coefficients (corresponding to ${\mathbf \Gamma} $)\\
$tol>0$ tolerance for the approximation error (e.g. $tol = 10^{-10}$)\\
$mmax \in {\mathbb N}$ maximal order of polynomials in the rational function
\begin{description}
\item{Initialization:}
Build $\tilde{\mathbf c} = \textrm{Real} ({\mathbf c}) + {\mathrm i} \, \textrm{Imag}({\mathbf c})./{\mathbf \Gamma}$.\\
Initialize the vectors ${\mathbf S}= []$, $\tilde{\mathbf c}_{S}=[]$. 
\begin{enumerate}
\item Find the two components $\tilde{c}_{n_{1}}$, $\tilde{c}_{n_{2}}$  of $\tilde{\mathbf c}$ with largest absolute values.\\
 Update ${\mathbf S},\, \tilde{\mathbf c}_{S}, \, {\mathbf \Gamma}, \, \tilde{\mathbf c}$ by adding $n_{1}$, $n_{2}$ as components of ${\mathbf S}$ and deleting these components in ${\mathbf \Gamma}$, adding $\tilde{c}_{n_{1}}$, $\tilde{c}_{n_{2}}$ in $\tilde{\mathbf c}_{S}$ and deleting them in $\tilde{\mathbf c}$. 
\item Compute ${\mathbf w} = (w_{1}, w_{2})^{T}$ via (\ref{ww}).
\item Compute ${\mathbf p} =  
{\mathbf C}_2 \left(\begin{array}{l} w_1 \, \tilde{c}_{n_1} \\ w_2 \, \tilde{c}_{n_2} \end{array} \right), \;
{\mathbf q} = {\mathbf C}_{2} \left(\begin{array}{l} w_1  \\ w_2  \end{array} \right)$, with 
${\mathbf C}_{2} = \left( \frac{1}{n^{2} - k^{2}} \right)_{n \in {\mathbf \Gamma}, k \in {\mathbf S}}$ 
and let ${\mathbf r} = (r_{j})_{j=1}^{L-2} = {\mathbf p} ./ {\mathbf q} \in {\mathbb C}^{L-2} $.
\item If $\| {\mathbf r} - \tilde{\mathbf c}\|_{\infty} < tol$ then stop.
\end{enumerate}
\item{Main Loop}\\
for $m=3:mmax$
\begin{enumerate}
\item Compute $k = \textrm{argmax}_{n \in {\mathbf \Gamma}} | r_{n} - \tilde{c}_{n} |$ and update  ${\mathbf S},\, \tilde{\mathbf c}_{S}, \, {\mathbf \Gamma}, \, \tilde{\mathbf c}$ by adding $k$ as a component of ${\mathbf S}$ and deleting it in ${\mathbf \Gamma}$, adding $\tilde{c}_{k}$ as a component of $\tilde{\mathbf c}_{S}$ and deleting it in $\tilde{\mathbf c}$.
\item Build the matrices ${\mathbf C}_{m} = \left( \frac{1}{n^{2} - k^{2}} \right)_{n \in {\mathbf \Gamma}, k \in {\mathbf S}}$ and 
${\mathbf A}_{m}= \left( \frac{\tilde{c}_{n} - \tilde{c}_{k}}{n^{2} - k^{2}} \right)_{n \in {\mathbf \Gamma}, k \in {\mathbf S}}$.
\item Compute the normalized right singular vectors ${\mathbf v}_{1}$ and ${\mathbf v}_{2}$ corresponding to the two smallest singular values of ${\mathbf A}_{m}$.
\item Compute ${\mathbf w} = ({\mathbf v}_{2}^{T} \tilde{\mathbf c}_{S}) \, {\mathbf v}_{1} - ({\mathbf v}_{1}^{T} \tilde{\mathbf c}_{S}) \, {\mathbf v}_{2}$ and normalize ${\mathbf w} = \frac{1}{\| {\mathbf w}\|_{2}} {\mathbf w}$.
\item  Compute ${\mathbf p} =  
{\mathbf C}_m \left({\mathbf w} .* \tilde{\mathbf c}_{S}  \right), \;
{\mathbf q} = {\mathbf C}_{m} \, {\mathbf w}$, and ${\mathbf r} = {\mathbf p}./ {\mathbf q} \in {\mathbb C}^{L-m}$.
\item  If $\| {\mathbf r} - \tilde{\mathbf c}\|_{\infty} < tol$ then stop.
\end{enumerate}
end(for)
\end{description}
\textbf{Output:} $K= m-1, \, {\mathbf S} \in {\mathbb N}^{K+1}, \, \tilde{\mathbf c}_{S} \in {\mathbb C}^{K+1}, \, {\mathbf w} \in {\mathbb C}^{K+1}$, where ${\mathbf S}$ is the vector of indices, $\tilde{\mathbf c}_{S}$ the corresponding coefficient vector of interpolation values, and ${\mathbf w}$ the weight vector to determine the rational function $r_{K}$ via (\ref{rat1}).
\end{algorithm}

 As we will show in Section \ref{sec:exact}, it will be sufficient to employ $L=2K+1$ Fourier  coefficients if the function $f$ has no $P$-periodic components $f_{j}$.
In  Section \ref{sec:periodic} we will prove that the algorithm can be also applied  if $f$ contains  $P$-periodic terms $f_{j}$ with frequencies $a_{j} \in \frac{1}{P}{\mathbb N}$. 
In this case we need  $L=2K+2$ Fourier  coefficients, and the rational output function will only determine the Fourier coefficients of the non-periodic part of $f$.

\subsection{Partial Fraction Representation of the Rational Approximant}
\label{sec:partial}

Assume that we have found the rational approximant $r_{K}(z)$ 
after $K$ iteration steps.
In other words,  we have now given the vector of interpolation indices $(n_1, \ldots , n_{K+1})^{T}$
and the weight vector $ {\mathbf w}_{K+1} = (w_{j})_{j=1}^{K+1} \in {\mathbb C}^{K+1}$, such that 
\begin{equation}\label{p1q1}
 r_{K}(z)  = \frac{\tilde{p}_K(z)}{\tilde{q}_K(z)} \qquad \textrm{with} \qquad  \tilde{p}_K(z) = \sum_{j=1}^{K+1} \frac{w_j \, \tilde{c}_{n_j}}{z-n_j^2} ,\qquad 
 \tilde{q}_K(z)  = \sum_{j=1}^{K+1} \frac{w_j}{z-n_j^2} .
 \end{equation}  
As we will show in Section \ref{sec:exact},  $r_{K}(z)$ determined by Algorithm \ref{alg1} coincides with the desired rational function in (\ref{tilde}) for exact data,  and, in particular, $w_{j} \neq 0$ for $j=1, \ldots , K+1$.
To extract the wanted parameters $A_j$, $B_j$, and $C_j$ in (\ref{cnrat})
from $\{n_1, \ldots , n_{K+1}\}$ and ${\mathbf w}_{K+1}$, we need to rephrase $r_{K}(z)$ in the form 
\begin{equation}\label{rk1} r_K(z) =   \sum_{j=1}^K \frac{A_j+ {\mathrm i}\, B_j }{z- C_j}. 
\end{equation}
The  parameters 
$C_j$, $j=1, \ldots , K$, are the zeros of the rational function $\tilde{q}_{K}(z)$ in (\ref{p1q1}), since $z=n_{j}^{2}$ cannot occur as poles of $r_{K}(z)$ due to the assumption $C_{j}^{1/2} \not\in {\mathbb N}_{0}$, which is by (\ref{cj}) equivalent with $a_{j} \not\in \frac{1}{P} {\mathbb N}_{0}$.

To compute the zeros of $\tilde{q}_K(z)$ we  again draw from the results in \cite{NST18} or \cite{Klein} and consider the generalized eigenvalue problem 
\begin{equation}\label{eig} \left( \begin{array}{ccccc}
0 & w_1 & w_2 & \ldots & w_{K+1} \\
1 & n_1^2 &   &     & \\
1 & & n_2^2 & & \\
\vdots & & & \ddots & \\
1 & & & & n_{K+1}^2 \end{array} \right) \, {\mathbf v}_{\lambda}= \lambda \left( \begin{array}{ccccc}
0 & & & & \\
& 1 & & & \\
& & 1 & & \\
& & & \ddots & \\
& & & & 1 \end{array} \right) \, {\mathbf v}_{\lambda}.
\end{equation}
Observe that $\lim_{z \to \pm \infty} \tilde{q}(z) = 0$ causes two infinite eigenvalues that we are not interested in.
The other $K$ eigenvalues are the wanted zeros of $\tilde{q}_K(z)$. This can be simply seen by taking the eigenvectors ${\mathbf v}_{\lambda}$ corresponding the eigenvalues $\lambda$ of the form 
$${\mathbf v}_{\lambda} = \left(1, \, \frac{1}{\lambda-n_1^2}, \, \ldots , \frac{1}{\lambda -n_{K+1}^2} \right)^T,$$
such that 
$$ \left( \begin{array}{cccc}
0 & w_1 & \ldots & w_{K+1} \\
1 & n_1^2 &       & \\
\vdots & &  \ddots & \\
1 & & & n_{K+1}^2 \end{array} \right)  {\mathbf v}_{\lambda}
= \left( \begin{array}{c} \sum\limits_{j=1}^{K+1} \frac{w_{j}}{\lambda-n_{j}^{2}} \\
1 + \frac{n_{1}^{2}}{\lambda- n_{1}^{2}} \\
\vdots \\ 
1 + \frac{n_{K+1}^{2}}{\lambda- n_{K+1}^{2}} \end{array}\right)  = 
\left(  \begin{array}{c}
\tilde{q}_{K}(\lambda) \\ \frac{\lambda}{\lambda - n_{1}^{2}} \\ \vdots  \\ \frac{\lambda}{\lambda-n_{K+1}^{2}} \end{array} \right) = \lambda 
\left( \begin{array}{c}
0 \\ \frac{1}{\lambda - n_{1}^{2}} \\ \vdots  \\ \frac{1}{\lambda-n_{K+1}^{2}} \end{array} \right).
$$
Having found the $K$ zeros $C_j$ of this eigenvalue problem, we obtain from the interpolation conditions $r_{K}(n_{\ell}^{2}) = \tilde{c}_{n_{\ell}}$ for $\ell=1, \ldots , K+1$ with (\ref{rk1}) the linear equation system 
$$ \sum_{j=1}^{K} \frac{A_{j} + {\mathrm i} B_{j}}{ n_{\ell}^{2} - C_{j}} = \tilde{c}_{n_{\ell}}, \qquad \ell =1, \ldots , K+1, $$
in order to determine $A_{j} + {\mathrm i} B_{j}$ $j=1, \ldots , K$.
We summarize the reconstruction of the parameter vectors $(A_{j})_{j=1}^{K}$, $(B_{j})_{j=1}^{K}$, and $(C_{j})_{j=1}^{K}$ from the output of Algorithm \ref{alg1} in Algorithm \ref{alg2}.

\begin{algorithm}[Reconstruction of parameters $A_{j}, \, B_{j}, \, C_{j}$ of partial fraction representation]
\label{alg2}
\textbf{Input:} ${\mathbf S} \in {\mathbb N}^{K+1}, \, \tilde{\mathbf c}_{S} \in {\mathbb C}^{K+1}, \, {\mathbf w} \in {\mathbb C}^{K+1}$, the output vectors of Algorithm \ref{alg1}
\begin{enumerate}
\item Build the matrices in (\ref{eig}) and solve this generalized eigenvalue problem to obtain the parameter vector $(C_{1}, \ldots, C_{K})^{T}$ of finite eigenvalues.
\item Build the matrix ${\mathbf V} = \left( \frac{1}{n^{2}- C_{j}}\right)_{n \in {\mathbf S}, j=1, \ldots , K} \in {\mathbb R}^{(K+1) \times K}$ and solve the linear system 
$$  {\mathbf V} \, {\mathbf x} = \tilde{\mathbf c}_{S}. $$
Set $(A_{j})_{j=1}^{K} = \textrm{Real} \, (\mathbf {x})$ and $(B_{j})_{j=1}^{K} = \textrm{Imag} \, (\mathbf{x})$.
\end{enumerate}
\textbf{Output:} Parameter vectors $(A_{j})_{j=1}^{K}$, $(B_{j})_{j=1}^{K}$, and $(C_{j})_{j=1}^{K}$.
\end{algorithm}

Finally, we can reconstruct the wanted parameter vectors $(\gamma_{j})_{j=1}^{K}$, $(a_{j})_{j=1}^{K}$ and $(b_{j})_{j=1}^{K}$ in (\ref{1.1}) via Theorem \ref{theo3}.

\section{Exact Reconstruction using  Algorithm \ref{alg1}}
\label{sec:exact}

Assume that  the Fourier coefficients $c_{n}(f)$ of a function $f$ in (\ref{1.1}) with $K$ components $f_j$ are given for $n \in \Gamma \subset {\mathbb N}$ with $\#\Gamma \ge 2K+1$, where we suppose that $a_j \not\in \frac{1}{P} {\mathbb Z}$ for all components $f_{j}$ of $f$.
We show that $f$ can be uniquely reconstructed from $2K$ Fourier coefficients $c_{n}(f)$.
Moreover, if at least $2K+1$ Fourier coefficients $c_{n}(f)$ are given, then Algorithm \ref{alg1} terminates after $K$ steps (i.e., taking $K+1$  interpolation points).

\begin{theorem}\label{theoexact}
Let $f$ be of the form $(\ref{1.1})$  
with $K \in {\mathbb N}$, $\gamma_{j} \in (0, \infty)$, 
and $(a_{j}, \, b_{j}) \in (0, \infty) \times  [0, \, 2\pi)$  
$j=1, \ldots , K$.
Further, let  $a_{j}$ be pairwise different and $a_{j}\not\in \frac{1}{P} {\mathbb  N}$ for a given $P>0$. Assume that we have a set of classical  Fourier coefficients $c_{n}(f)$, $n \in \Gamma \subset {\mathbb N}$ (with regard to period $P$) with $L = \# \Gamma \ge 2K+1$.  Then $f$ is uniquely determined by $2K$ of these Fourier coefficients and  Algorithm $\ref{alg1}$ terminates after $K$ steps taking $K+1$ interpolation points and determines the rational function $r_{K}(z) = \sum_{j=1}^{K} \frac{A_{j}+ {\mathrm i} B_{j}}{z-C_{j}}$ in $(\ref{tilde})$ that interpolates all $\tilde{c}_{n}(f)$, $n \in {\mathbb N}$,  exactly.
\end{theorem} 

\begin{proof}
1. From (\ref{tilde}) it follows that there exists a rational function $r_{K}(z) = \frac{p_{K-1}(z)}{q_{K}(z)}$ of type $(K-1, K)$ such that $\tilde{c}_{n}= \tilde{c}_{n}(f) = r_{K}(n^{2}) = \frac{p_{K-1}(n^{2})}{q_{K}(n^{2})}$ for all $n \in {\mathbb N}$ with $q_{K}(z)$ in (\ref{qform}) and $p_{K-1}(z)$ in (\ref{pform}). In particular, $p_{K-1}(z)$ and $q_{k}(z)$ are coprime.

First, we show that $r_{K}(z)$ is uniquely determined by $2K$ coefficients $\tilde{c}_{n_{j}}(f)$, $n_{j} \in \Gamma$, $j=1, \ldots , 2K$.
By (\ref{pform}), 
$p_{K-1}(z)$ has at most  degree $K-1$. Further, $q_{K}(z)$ has exactly degree $K$ by (\ref{qform}).
We use the notation $p_{K-1}(z)= \sum_{r=0}^{K-1} p_{r} z^{r}$ and $q_{K}= z^{K} + \sum_{r=0}^{K-1}q_{r} \, z^{r}$. Then the interpolation conditions 
$$  r_{K}(n_{j}^{2}) = \tilde{c}_{n_{j}}, \qquad j=1, \ldots , 2K,$$
yield the equation system 
$$ -\sum_{r=0}^{K-1} p_{r} \, n_{j}^{2r} + \tilde{c}_{n_{j}} \left( n_{j}^{2K} + \sum_{r=0}^{K-1} q_{r} \, n_{j}^{2r} \right)   =0, \qquad j=1, \ldots 2K. $$
This leads to the homogeneous system 
\begin{equation}\label{Wsystem} {\mathbf W} \left( \begin{array}{c} \!\!{\mathbf p}\!\! \\ \!\!{\mathbf q}\!\! \end{array} \right) = {\mathbf 0} 
\end{equation}
with the coefficient matrix
$$ {\mathbf W} = \left( \begin{array}{cccccccc}
\!\!-n_{1}^{0} & -n_{1}^{2}& \!\!\ldots\!\! & -n_{1}^{2K-2} \, & \tilde{c}_{n_{1}}n_{1}^{0} & \tilde{c}_{n_{1}} n_{1}^{2} & \!\!\ldots\!\! &
\tilde{c}_{n_{1}} n_{1}^{2K} \\
\!\!-n_{2}^{0} & -n_{2}^{2}& \!\!\ldots\!\! & -n_{2}^{2K-2} \, & \tilde{c}_{n_{2}}n_{2}^{0} & \tilde{c}_{n_{2}} n_{2}^{2} & \!\!\ldots\!\! &
\tilde{c}_{n_{2}} n_{2}^{2K} \\
\vdots & \vdots & & \vdots & \vdots & \vdots & & \vdots \\
\!\!-n_{2K}^{0} & -n_{2K}^{2}& \!\!\ldots\!\! & -n_{2K}^{2K-2} \, & \tilde{c}_{n_{2K}}n_{2K}^{0} & \tilde{c}_{n_{2K}} n_{2K}^{2} & \!\!\ldots\!\! &
\tilde{c}_{n_{2K}} n_{2K}^{2K} \end{array} \right) \in {\mathbb R}^{2K \times (2K+1)}
$$
and with 
$({\mathbf p}^{T}, \, {\mathbf q}^{T}) = (p_{0}, \ldots , p_{K-1}, q_{0}, \ldots q_{K-1},1)^{T} \in {\mathbb R}^{2K+1}$.

The kernel of ${\mathbf W}$  has at least dimension $1$, and by construction, the vector $({\mathbf p}^{T}, \, {\mathbf q}^{T})$ 
generating the  rational function $r_{K}(z) = p_{K-1}(z) / q_{K}(z)$  satisfies (\ref{Wsystem}).   We show that the kernel of ${\mathbf W}$ has exactly dimension $1$.  Suppose to the contrary that there exists  another vector in the kernel  of ${\mathbf W}$ being linearly independent  of $({\mathbf p}^{T}, \, {\mathbf q}^{T})$. 
Then  we also find  a kernel vector, whose last component vanishes, i.e., of the form 
$
(\breve{\mathbf p}^{T}, \, \breve{\mathbf q}^{T}) = (\breve{p}_{0}, \ldots , \breve{p}_{K-1}, \breve{q}_{0}, \ldots \breve{q}_{K-2}, \breve{q}_{K-1}, 0)^{T}$. 
Thus, there exist polynomials $\breve{p}_{K-1}$ and $\breve{q}_{K-1}$ of at most degree $K-1$  satisfying 
$$ \breve{p}_{K-1}(n_{j}^{2}) = \tilde{c}_{n_{j}}(f) \, \breve{q}_{K-1}(n_{j}^{2}), \qquad j=1, \ldots , 2K.$$
Using the known structure of $\tilde{c}_{n_{j}}(f) = p_{K-1}(n_{j}^{2})/q_{K}(n_{j}^{2}) $ in (\ref{tilde}), we obtain 
$$ \breve{p}_{K-1}(n_{j}^{2})  \, q_{K}(n_{j}^{2}) - p_{K-1}(n_{j}^{2}) \, \breve{q}_{K-1}(n_{j}^{2}) = 0, \qquad j=1, \ldots , 2K.$$
Since the degree of the involved polynomial products is at most $2K-1$, 
it follows that $\breve{p}_{K-1}(z)  \, q_{K}(z) = p_{K-1}(z) \, \breve{q}_{K-1}(z)$ for all $z \in {\mathbb R}$.
But $q_{K}(z)$ is a monic polynomial of degree $K$ and the polynomials $p_{K-1}(z)$, $q_{K}(z)$ are coprime,  and we conclude that $\breve{q}_{K-1}(z)$ possesses all $K$ linear factors of $q_{K}(z)$. This leads to a contradiction, since $\breve{q}_{K-1}(z)$ has degree at most $K-1$.
Therefore, there exists only one normalized solution vector of the form (\ref{Wsystem}),
which is already uniquely defined by $2K$ modified Fourier coefficients of $f$, 
and this solution vector $({\mathbf p}^{T}, \, {\mathbf q}^{T})^{T}$  determines the rational polynomial $r_{K}(z) = \frac{p_{K-1}(z)}{q_{K}(z)}$ that satisfies all $2K$ interpolation conditions. 

2. We show now that  Algorithm \ref{alg1} leads to this unique solution $r_{K}(z)$ after $K$ steps.
Assume that $\Gamma$ contains $L \ge 2K+1$ indices. 
At the $K$-th  iteration step we have chosen a set $S_{K+1}$ of $K+1$ pairwise different indices $n_{\ell} \in \Gamma$ for interpolation and start with the ansatz
\begin{equation}\label{ansatz}
 r(z) = \frac{\sum_{\ell=1}^{K+1} \frac{w_{\ell} \, \tilde{c}_{n_{\ell}}}{z - n_{\ell}^{2}}}{\sum_{\ell=1}^{K+1} \frac{w_{\ell}}{z-n_{\ell}^{2}}}, 
\end{equation} 
such that the interpolation conditions $r(n_{\ell}^{2}) = \tilde{c}_{n_{\ell}}$ are already satisfied for   $\ell=1, \ldots , K+1$,  if $w_{\ell} \neq 0$. 
Let  $\Gamma_{K+1} := \Gamma \setminus S_{K+1}$.
Now,  Algorithm  \ref{alg1} 
determines the weight vector ${\mathbf w}= (w_{\ell})_{\ell=1}^{K+1}$ 
as a linear combination of the two right singular vectors ${\mathbf v}_{1}$ and ${\mathbf v}_{2}$ of 
$$ {\mathbf A}_{K+1} := \left( \frac{\tilde{c}_{n} - \tilde{c}_{n_{\ell}}}{n^{2} - n_{\ell}^{2}}
 \right)_{n\in \Gamma_{K+1}, n_{\ell} \in S_{K+1}} \in {\mathbb C}^{(L-K-1) \times (K+1)}$$
corresponding  to the two smallest singular values $\sigma_{1} \le \sigma_{2}$ 
 with side conditions $\|{\mathbf w}\|_{2} =1$ and $\sum_{\ell=1}^{K+1} w_{\ell} \tilde{c}_{n_{\ell}} =0$.
From (\ref{cnrat}) and (\ref{tilde}) it follows that 
\begin{eqnarray*}
{\mathbf A}_{K+1} \!\!\!&=&  \!\!\! \left( \frac{ \sum\limits_{j=1}^{K} (A_{j} + {\mathrm i} B_{j}) \Big( \frac{1}{n^{2}- C_{j}} - \frac{1}{n_{\ell}^{2} - C_{j}} \Big)}{n^{2}- n_{\ell}^{2}} \right)_{n\in \Gamma_{K+1}, n_{\ell} \in S_{K+1}}\\
 &=&\!\!\!\left(  \sum_{j=1}^{K} \frac{-(A_{j} + {\mathrm i} B_{j})}{(n^{2}-C_{j})(n_{\ell}^{2} - C_{j})} \right)_{n\in \Gamma_{K+1}, n_{\ell} \in S_{K+1}} \\
 &=& \!\!\!\left( \frac{1}{n^{2} - C_{j}} \right)_{n \in \Gamma_{K+1}, j=1, \ldots , K} 
 \mathrm{diag}   \left( (-(A_{j} + {\mathrm i} B_{j}))_{j=1}^{K} \right)   \left( \frac{1}{n_{\ell}^{2} - C_{j}} \right)_{j=1,\ldots, K,  n_{\ell} \in S_{K+1}},
\end{eqnarray*}
where the two Cauchy matrices have full rank $K$ and where the entries $-(A_{j} + {\mathrm i} B_{j})$ of the diagonal matrix do not vanish for all $j=1, \ldots , K$.
Thus,  ${\mathbf A}_{K+1}$ has exactly rank $K$ and therefore a kernel of dimension $1$.  Let ${\mathbf v}_{1}$  be the normalized right singular vector of ${\mathbf A}_{K+1}$ to $\sigma_{1}=0$, i.e., ${\mathbf A}_{K+1} {\mathbf v}_{1}={\mathbf 0}$.
The factorization of ${\mathbf A}_{K+1}$ also implies 
$$ \left( \frac{1}{n_{\ell}^{2} - C_{j}} \right)_{j=1, \ldots , K, n_{\ell} \in S_{K+1}} {\mathbf v}_{1} = {\mathbf 0}. $$
We observe that, if ${\mathbf v}_{1} \neq {\mathbf 0}$ had one or more vanishing components, then $K$ columns of the Cauchy matrix $\left( \frac{1}{n_{\ell}^{2} - C_{j}} \right)_{j=1, \ldots , K, n_{\ell} \in S_{K+1}}$ would be  linearly dependent. But this is not possible, since $C_{j} \not\in {\mathbb N}$ are pairwise distinct and therefore any $K$ columns of this Cauchy matrix are linearly independent.
Thus, all components of ${\mathbf v}_{1}$ are nonzero. 

If we determine the rational polynomial $r(z)$ in (\ref{ansatz}) with ${\mathbf w}:={\mathbf v}_{1}$, then it follows  $r(n_{\ell}^{2}) = \tilde{c}_{n_{\ell}}$ for $n_{\ell} \in S_{K+1}$ by construction, since all weight components  are nonzero.
Moreover,  the condition ${\mathbf A}_{K+1} {\mathbf w} = {\mathbf 0}$ leads to
$$ \sum_{\ell=1}^{K+1} \left(  \frac{\tilde{c}_{n} w_{\ell}}{n^{2} - n_{\ell}^{2}} - \frac{\tilde{c}_{n_{\ell}} w_{\ell}}{n^{2} - n_{\ell}^{2}}\right) = 0, \qquad n \in {\Gamma}_{K+1},  $$
i.e., it follows that $r(n^{2}) = \tilde{c}_{n}$ for all $n \in \Gamma_{K+1}$.
Thus, the first part of the proof implies that the obtained rational function $r(z)$ in (\ref{ansatz}) coincides with $r_{K}(z) = \frac{\tilde{p}_{K}(z)}{\tilde{q}_{K}(z)}$, and
therefore has to be of the wanted type $(K-1, K)$ since it is already uniquely defined by the interpolation conditions.
We conclude that ${\mathbf w}={\mathbf v}_{1}$  already satisfies the side condition $\sum_{\ell=1}^{K+1} w_{\ell} \tilde{c}_{n_{\ell}} =0$ and is therefore the weight vector  computed at the $K$-th iteration step of Algorithm \ref{alg1}.
\end{proof}

\section{How to Proceed if the Function Contains $P$-Periodic Terms}
\label{sec:periodic}

Let us assume that the function  $f(t) = \sum_{j=1}^{K} f_{j}(t)$ is of the form  (\ref{1.1}) with $K$ components $f_j$, where beside non-periodic components $f_j(t) = \gamma_{j} \cos(2\pi a_{j} t + b_{j})$ with $a_j \not\in \frac{1}{P} {\mathbb N}$, there are also periodic components with $a_j \in \frac{1}{P} {\mathbb N}$.
Now, we will study the question, how Algorithm \ref{alg1} proposed in Section \ref{sec:AAA} behaves in this case  and how we can reconstruct $f(t)$.

We assume that the index set $\Gamma$ of given Fourier coefficients of $f$ contains all integers $n_j$, if  $a_j = \frac{n_j}{P}$  occurs as a frequency in a component $f_j$ of $f$. Otherwise, the component $f_j$ cannot be identified from the given data.
We assume further that  $L= \# \Gamma  \ge 2K+2$.
The function $f$ in (\ref{1.1}) (with the usual restrictions $\gamma_{j} \in (0, \infty)$, $(a_{j}, \, b_{j}) \in (0, \infty) \times [0, 2\pi)$) can now be written as $f(t)= \phi_{1}(t) + \phi_{2}(t)$, where
\begin{equation}\label{phi1} 
\phi_{1}(t) = \sum_{j=1}^{K_{1}} \gamma_{j} \, \cos(2\pi a_{j} t + b_{j}), \qquad K_{1} < K,\; a_{j} \not\in \frac{1}{P} {\mathbb N},
\end{equation}
is non-P-periodic such that the modified Fourier coefficients  $\tilde{c}_{n}(\phi_{1}) = \mathrm{Re} \, c_{n}(\phi_{1}) + \frac{{\mathrm i}}{n} \, \textrm{Im} \, c_{n}(\phi_{1})$ are of the form
\begin{equation}\label{cnphi} \tilde{c}_{n}(\phi_{1}) = \sum_{j=1}^{K_{1}} \frac{A_{j} + {\mathrm i}  B_{j}}{n^{2} - C_{j}} = \frac{p_{K_{1}-1}(n^{2})}{q_{K_{1}}(n^{2})} = r_{K_{1}}(n^{2}), 
\end{equation}
similarly as in (\ref{tilde}). The $P$-periodic  part of $f(t)$, 
\begin{equation}\label{phi2}
\phi_{2}(t) = \sum_{j=K_{1}+1}^{K} \gamma_{j} \, \cos( 2\pi a_{j} t + b_{j}) \quad \mathrm{with} \quad   a_{j} \in \frac{1}{P}{\mathbb N},
\end{equation}
has only $K_{2}:=K-K_{1}$ nonzero Fourier coefficients with non-negative index. Let $\Sigma:=\{ P \, a_{K_{1}+1}, \ldots , P\, a_{K} \} \subset {\mathbb N}$ denote the corresponding  index set. Then $K_{2} = \# \Sigma$, and 
\begin{equation}\label{phi2c}
\tilde{c}_{n}(\phi_{2}) = \left\{ \begin{array}{ll} \frac{\gamma_{j}}{2} (\cos b_{j} - \frac{{\mathrm i}}{n} \sin b_{j}) & n \in \Sigma,\\
0 & n \not\in \Sigma. \end{array} \right.
\end{equation} 
Since $\tilde{c}_{n}(f) =  \tilde{c}_{n}(\phi_{1}) + \tilde{c}_{n}(\phi_{2})$ for $n \in {\mathbb N}$, it follows that only $K_{2}$  modified Fourier coefficients of $f$ are not of the  form as in (\ref{tilde}) while all $\tilde{c}_{n}(f)$ with $n \not\in \Sigma$ satisfy $\tilde{c}_{n}(f) = \tilde{c}_{n}(\phi_{1})$ and can be reconstructed by a rational function of type $(K_{1}-1, K_{1})$. Let us now examine, how to reconstruct  $f(t)$  in this setting.

\begin{theorem}\label{theoperiodic}
Assume that $f(t)$ is of the form $f(t) = \phi_{1}(t) + \phi_{2}(t)$ with $\phi_{1}(t)$ and $\phi_{2}(t)$ in $(\ref{phi1})$ and $(\ref{phi2})$, where $\phi_{2}$ possesses the $K_{2}$ nonzero Fourier coefficients $c_{n}(\phi_{2})$, $n \in \Sigma \subset {\mathbb N}$.
Assume that we have given a set ${c}_{n}(f)$, $n \in \Gamma \subset {\mathbb N}$ of $L\ge 2K+2$ Fourier coefficients of $f$,  where the unknown set $\Sigma$ is contained in  $\Gamma$. 
Then  $\phi_{1}$ and $\phi_{2}$, i.e., $K_{1}, \, K_{2}$ as well as all parameters $\gamma_{j}, \, a_{j}, \, b_{j}$  determining $\phi_{1}$ and $\phi_{2}$ can be completely recovered from this set of Fourier coefficients. In particular, 
Algorithm $\ref{alg1}$ terminates after at most $K+1$ steps and provides a rational function $r_{K_{1}}(z)$  of type $(K_{1}-1,K_{1})$ that interpolates all  $\tilde{c}_{n}(f) = \tilde{c}_{n}(\phi_{1})$ for all $n \in {\mathbb N} \setminus \Sigma$.
\end{theorem}

\begin{proof}
Let as before $\tilde{c}_{n}:= \mathrm{Re} \, c_{n}(f) + \frac{{\mathrm i}}{n} \, \textrm{Im} \, c_{n}(f)$.
Assume that we have taken a set $S_{K+2} \subset \Gamma$ of $K+2$ given indices as interpolation points  at the $(K+1)$-th iteration step of Algorithm \ref{alg1}. We will show that the matrix ${\mathbf A}_{K+2} = \left( \frac{\tilde{c}_{n} - \tilde{c}_{k}}{n^{2} - k^{2}} \right)_{n \in \Gamma_{K+2}, k \in S_{K+2}} \in {\mathbb C}^{L-K-2 \times K+2}$ has rank $K$ and possesses  a kernel vector  ${\mathbf w} \in {\mathbb C}^{K+2}$  that satisfies the side condition  ${\mathbf w}^{T} \tilde{\mathbf c}_{S_{K+2}} = 0$. Hence, we can show that the weight vector ${\mathbf w}$ found in Algorithm \ref{alg1} determines  the rational function  $r_{K_{1}}(z)$  which  interpolates $\tilde{c}_{n}(\phi_{1})$ for all $n \in {\mathbb N}$.

1. Let $\Sigma' \cup \Sigma'' = \Sigma$ with $\Sigma' \cap \Sigma'' = \emptyset$  be the partition of $\Sigma$ such that $\Sigma' \subset S_{K+2}$ and $\Sigma'' \subset \Gamma_{K+2}$, and let $K_{2}' := \# \Sigma'$ and $K_{2}'' := \# \Sigma''$ denote the numbers of elements of $\Sigma'$ and $\Sigma''$, such that $K_{1}+ K_{2}' + K_{2}'' = K$.
Then the $K+2-K_{2}'$  indices in $S_{K+2} \setminus \Sigma'$ correspond to  modified Fourier coefficients with the rational structure $\tilde{c}_{n}(f) = \tilde{c}_{n}(\phi_{1}) = \frac{p_{K_{1}-1}(n^{2})}{q_{K_{1}}(n^{2})} = r_{K_{1}}(n^{2})$ as in (\ref{cnphi}), and the same is true for the $L-K-2-K_{2}''$ indices in  $\Gamma_{K+2} \setminus \Sigma''$.

Assume that the rows and columns of the matrix ${\mathbf A}_{K+2}$ are ordered such that the first $K+2-K_{2}'$ columns of ${\mathbf A}_{K+2}$  correspond  to  $S_{K+2}\setminus \Sigma'$, while the last $K_{2}'$ columns  correspond to the index set $\Sigma'$. Similarly,  we suppose that  the rows of ${\mathbf A}_{K+2}$ are ordered  such that  the first  $L-K-2-K_{2}''$ rows  correspond to the indices  $\Gamma_{K+2} \setminus \Sigma''
$, while the remaining  $K_{2}''$ rows  correspond to  indices in $\Sigma''$.  In other words, we  obtain 
$$ {\mathbf A}_{K+2} = \left( \begin{array}{ll} {\mathbf A}_{11} & {\mathbf A}_{12} \\
{\mathbf A}_{21} & {\mathbf A}_{22} \end{array} \right) $$
with 
\begin{eqnarray*}
{\mathbf A}_{11} &=& \left( \frac{\tilde{c}_{n}(\phi_{1}) - \tilde{c}_{k}(\phi_{1})}{n^{2} - k^{2}} \right)_{n \in \Gamma_{K+2} \setminus \Sigma'',\,  k \in S_{K+2} \setminus \Sigma'} \in {\mathbb C}^{(L-K-2-K_{2}'') \times (K+2-K_{2}')}, \\
{\mathbf A}_{12} &=& \left( \frac{\tilde{c}_{n}(\phi_{1}) - \tilde{c}_{k'}(\phi_{1})}{n^{2} - (k')^{2}} - \frac{\tilde{c}_{k'}(\phi_{2})}{n^{2}- (k')^{2}}\right)_{n \in \Gamma_{K+2} \setminus \Sigma'',\,  k' \in \Sigma'} \in {\mathbb C}^{(L-K-2-K_{2}'') \times K_{2}'}, \\
{\mathbf A}_{21} &=& \left( \frac{\tilde{c}_{n'}(\phi_{1}) - \tilde{c}_{k}(\phi_{1})}{(n')^{2} - k^{2}} + \frac{\tilde{c}_{n'}(\phi_{2})}{(n')^{2}- k^{2}}\right)_{n' \in  \Sigma'',\,  k \in S_{K+2} \setminus \Sigma'} \in {\mathbb C}^{K_{2}'' \times (K+2-K_{2}')}, \\
{\mathbf A}_{22} &=& \left( \frac{\tilde{c}_{n'}(\phi_{1}) - \tilde{c}_{k'}(\phi_{1})}{(n')^{2} - (k')^{2}} + \frac{\tilde{c}_{n'}(\phi_{2})- \tilde{c}_{k'}(\phi_{2})}{(n')^{2}- (k')^{2}}\right)_{n' \in  \Sigma'', \, k' \in \Sigma'} \in {\mathbb C}^{K_{2}'' \times K_{2}'}. 
\end{eqnarray*}
Since ${\mathbf A}_{11}$ is only composed of the modified Fourier coefficients  of the non-periodic function $\phi_{1}$, it follows similarly as in the proof of Theorem \ref{theoexact} that ${\mathbf A}_{11}$ possesses rank $K_{1}$. More exactly, we have with (\ref{cnphi}) the matrix factorization  
$$ \textstyle {\mathbf A}_{11} = \left( \frac{1}{n^{2}- C_{j}} \right)_{n\in \Gamma_{K+2} \setminus \Sigma'', \, j=1,\ldots, K_{1}}  \mathrm{diag} \left( \!\Big( -(A_{j} + {\mathrm i} B_{j}) \Big)_{j=1}^{K_{1}} \!\right) \, 
 \left( \frac{1}{k^{2}- C_{j}} \right)_{j=1, \ldots , K_{1}, k \in S_{K+2} \setminus \Sigma'}, $$
 where the two Cauchy matrices and the diagonal matrix have full rank $K_{1}$.  
Therefore,  ${\mathbf A}_{11}$ possesses a kernel of dimension $K+2-K_{1}-K_{2}' = K_{2}''+2$. 
Since  ${\mathbf A}_{21}$ contains only $K_{2}''
\le K_{2}$ rows, it follows that $\left( \begin{array}{ll} \!\! {\mathbf A}_{11}\!\!\\ \!\!{\mathbf A}_{21} \!\!\end{array} \right)$ has at most rank $K_{1}+K_{2}''$ and therefore possesses a kernel of dimension at least $2$.
Thus, ${\mathbf A}_{K+2}$ has at most rank $K_{1}+ K_{2}''+ K_{2}'= K$.  

2. 
We will prove that ${\mathbf A}_{K+2}$  has exactly rank $K$ by showing that 
rank $({\mathbf A}_{11}, {\mathbf A}_{12}) = K_{1}+K_{2}'$
and similarly, 
that rank $\left(\begin{array}{l} \!\!{\mathbf A}_{11} \!\!\\ \!\!{\mathbf A}_{21}\!\! \end{array} \right) = K_{1}+K_{2}''$.
As in the proof of Theorem 4.1, we can  always find a linear combination  of $K_{1}$ columns  of ${\mathbf A}_{11}$  to represent the columns $\Big( \frac{\tilde{c}_{n}(\phi_{1}) - \tilde{c}_{k'}(\phi_{1})}{n^{2} - k'^{2}} \Big)_{n \in \Gamma_{K+2} \setminus \Sigma''}$ for all $k' \in \Sigma'$.  
Indeed for each of these columns we have
\begin{eqnarray*}
\textstyle \Big(  \frac{\tilde{c}_{n}(\phi_{1}) - \tilde{c}_{k'}(\phi_{1})}{n^{2}- k'^{2}} \Big)_{n \in \Gamma_{K+2} \setminus \Sigma''} 
\!\!\!\!&=& \!\!\!\!\textstyle \Big( \frac{1}{n^{2}- C_{j}} \Big)_{n\in \Gamma_{K+2} \setminus \Sigma'',  j=1,\ldots, K_{1}}  \mathrm{diag} \Big(\!\!-\!(A_{j} \!+\! {\mathrm i} B_{j}) \Big)_{j=1}^{K_{1}} 
 \Big( \frac{1}{k'^{2}- C_{j}} \Big)_{j=1}^{K_{1}}, 
 \end{eqnarray*}
 where $\left( \frac{1}{k'^{2}- C_{j}} \right)_{j=1}^{K_{1}}$ can be written as linear combination of the columns in \newline
 $ \left( \frac{1}{k^{2}- C_{j}} \right)_{j=1, \ldots , K_{1}, k \in S_{K+2} \setminus \Sigma'}$ which generate ${\mathbb R}^{K_{1}}$. 
Thus, 
$$ \mathrm{rank} \, \left( {\mathbf A}_{11}, \, {\mathbf A}_{12} \right) = \mathrm{rank} \, \left( {\mathbf A}_{11}, \, -\left(\frac{\tilde{c}_{k'}(\phi_{2})}{n^{2}- k'^{2}}\right)_{n \in \Gamma_{K+2} \setminus \Sigma'',\,  k' \in \Sigma'} \right), $$
where
$$ \left(  \frac{\tilde{c}_{k'}(\phi_{2})}{n^{2}- k'^{2}}\right)_{n \in \Gamma_{K+2} \setminus \Sigma'', k' \in \Sigma'}
= \left(  \frac{1}{n^{2}- k'^{2}}\right)_{n \in \Gamma_{K+2} \setminus \Sigma'', k' \in \Sigma'} \, \mathrm{diag} \, \left( \tilde{c}_{k'}(\phi_{2}) \right)_{k' \in \Sigma'} , $$
and  $ \tilde{c}_{k'}(\phi_{2}) \neq 0$ for  $k' \in \Sigma'$.
Therefore, it suffices to show that the concatenation of the first matrix factor of ${\mathbf A}_{11}$ and 
$\left( \frac{1}{n^{2}- k'^{2}} \right)_{n \in \Gamma_{K+2} \setminus \Sigma'', k' \in \Sigma'}$,  i.e., 
$$ \left( \left(\frac{1}{n^{2}- C_{j}} \right)_{n\in \Gamma_{K+2} \setminus \Sigma'', \, j=1,\ldots, K_{1}} , \left( \frac{1}{n^{2}- k'^{2}} \right)_{n \in \Gamma_{K+2} \setminus \Sigma'', k' \in \Sigma'} \right)
$$
has full rank $K_{1}+ K_{2}'$. This is obviously true since this Cauchy matrix has $L-K-2-K_{2}'' \ge K-K_{2}''=K_{1}+K_{2}'$ rows and the values $C_{j}$, $j=1, \ldots , K_{1}$ and $k'^{2}$, $k'  \in \Sigma'$ are pairwise distinct and also distinct from $n^{2}$ with $n  \in \Gamma_{K+2}\setminus \Sigma''$.
Similarly we can show that rank $\left({\mathbf A}_{11}^{T}, \, {\mathbf A}_{21}^{T} \right)$  can be simplified to 
$$ \mathrm{rank}  \left( {\mathbf A}_{11}^{T}, \, \left(  \frac{\tilde{c}_{n'}(\phi_{2})}{n'^{2}- k^{2}}\right)_{n' \in  \Sigma'', k \in S_{K+2} \setminus \Sigma'}^{T} \right)
$$
and has rank $K_{1}+K_{2}''$.

4. Thus rank ${\mathbf A}_{K+2} = K$, i.e., the dimension of the kernel of ${\mathbf A}_{K+2}$ is $2$.
Therefore, Algorithm \ref{alg1}  always finds  a vector ${\mathbf w}$ in the kernel of ${\mathbf A}_{K+2}$  which satisfies  also the side condition ${\mathbf w}^{T} \tilde{\mathbf c}_{S_{K+2}} =0$.
Moreover, it follows from the previous observations that any vector ${\mathbf w}$ in the kernel of  ${\mathbf A}_{K+2}$, is of the form 
$ {\mathbf w} = (\tilde{\mathbf w}^{T}, {\mathbf 0}^{T}) \in {\mathbb C}^{K+1}$, where  $\tilde{\mathbf w} \in {\mathbb C}^{K+2-K_{2}'}$ is in the kernel of $\left( \begin{array}{c} \!\!\!{\mathbf A}_{11} \!\!\!\\ \!\!\!{\mathbf A}_{21} \!\!\!\end{array} \right)$
and particularly in the kernel of ${\mathbf A}_{11}$. Thus, it follows from Theorem \ref{theoexact}  that $\tilde{\mathbf w}$  has at least $K_{1}+1$ nonzero components and provides  the rational  function  $r_{K_{1}}(z)$ that interpolates  all modified Fourier coefficients of $\phi_{1}$.  However, differently from the proof of  Theorem \ref{theoexact}, $\tilde{\mathbf w}$ may possess more than $K_{1}+1$  nonzero components, and the computation of $r_{K_{1}}(z)$ may involve the removal of Froissart doublets. 

5. 
Having  determined  $r_{K_{1}}(z)$ to interpolate all  modified Fourier coefficients of  $\phi_{1}$, we can find  $\phi_{2}$ of the form (\ref{phi2}) by capturing  all modified Fourier coefficients $\tilde{c}_{n}(f)$  with $\tilde{c}_{n}(f) \neq \tilde{c}_{n}(\phi_{1})$, $n  \in \Gamma$.
For all $n \in \Gamma$, we compute $\tilde{c}_{n}(\phi_{2}) = \tilde{c}_{n}(f) - r_{K_{1}}(n^{2})$.
Then $\phi_{2}$ can be reconstructed from (\ref{phi2}) and (\ref{phi2c}), where $\Sigma$ is found as the set of indices $n \in \Gamma$ with $\tilde{c}_{n}(\phi_{2}) \neq 0$, and $K_{2}= \# \Sigma$.
\end{proof}

\begin{remark}
We can also reconstruct $\tilde{f}(t) = \gamma_{0} + f(t) = \gamma_{0} + \sum_{j=1}^{K} f_{j}(t)$, where $\gamma_{0} \neq 0$ is a constant. Then $\gamma_{0}$ can be  seen as a periodic component of the function, and we have $c_{n}(\tilde{f}) = c_{n}(f)$ for all $n \neq 0$. Thus, if beside the set of Fourier coefficients $c_{n}(f)$, $n \in \Sigma$, in Theorem \ref{theoperiodic}
also  $c_{0}(\tilde{f})$ is known, then the constant part $\gamma_{0}$ can be reconstructed, too.
\end{remark}

\section{Generalization of the Model}
\label{sec:general}

The model (\ref{1.1}) for signals considered in the previous sections can be generalized.
Beside $(a_{j}, \, b_{j}) \in (0, \infty)  \times  [0, \, 2\pi)$ 
we now admit $ (a_{j}, \, b_{j}) \in {\mathrm i} (0, \infty) \times {\mathrm i} {\mathbb R}$ for $j \in \{1, \ldots , K\}$, i.e., the parameters $a_{j}$ and $b_{j}$ are complex with vanishing real part.
Observing that $\cos({\mathrm i} (2\pi a t + b) ) = \cosh(2\pi a t + b)$ for real numbers $a, \, b$,  we can consider the  generalized function model 
\begin{equation}\label{1.1a}
f(t) = \sum_{j=1}^{K} \, \gamma_j \, \cos(2 \pi a_j  t + b_j) = \sum_{j=1}^{\kappa} \, \gamma_j \, \cos(2 \pi a_j  t + b_j) + \sum_{j=\kappa+1}^{K} \, \, {\gamma}_j \, \cosh(2\pi \tilde{a}_j t + \tilde{b}_j)
\end{equation}
where $\kappa \in \{1, \ldots , K\}$,  $\gamma_j \in (0, \infty)$  for $j=1, \ldots , K$,  and $(a_j, \, b_{j}) \in (0, \infty) \times [0, 2\pi)$ 
for $j=1, \ldots , \kappa$, as well as $\tilde{a}_{j}= -{\mathrm i} \, a_{j} \in (0, \infty)$,  $\tilde{b}_{j} = -{\mathrm i} \, b_{j}\in {\mathbb R}$ for $j=\kappa+1, \ldots , K$. 
Here, we assume as before that $a_{j}$, $j=1,\ldots K$, are pairwise distinct. 
Model (\ref{1.1}) is obtained from (\ref{1.1a}) for  $\kappa=K$.
In particular, we obtain similarly to Theorem \ref{theo1} the uniqueness of the parameter representation  of $f$ in (\ref{1.1a}).
\begin{corollary}\label{coro1}
Let $f$ be given as in $(\ref{1.1a})$ with $K\in {\mathbb N}$, $\gamma_j \in (0, \infty)$, and $(a_j , \, b_j) \in (0, \infty) \times [0, 2\pi)$, or  $(a_j , \, b_j) \in {\mathrm i} (0, \, \infty) \times {\mathrm i} {\mathbb R}$,
where $a_{j}$ are pairwise distinct. Further, let 
$$ g(t) = \sum_{j=1}^{M} \delta_{j} \, \cos(2 \pi c_{j}  t + d_{j}) $$
with $M \in {\mathbb N}$, $\delta_j \in (0, \infty)$,  and 
$(c_j , \, d_j) \in (0, \infty) \times [0, 2\pi)$ or  $(c_j , \, d_j) \in {\mathrm i} (0, \, \infty) \times {\mathrm i} {\mathbb R}$,
where $c_{j}$, $j=1, \ldots , M$ are pairwise distinct. 
If $f(t) = g(t)$ for all $t$ on an interval $T \subset {\mathbb R}$ of positive length,  then we have 
$K=M$ and (after suitable permutation of the summands) $\gamma_{j}=\delta_{j}$, $a_{j}=c_{j}$, $b_{j}= d_{j}$ for $j=1, \ldots , K$.
\end{corollary}

Corollary \ref{coro1} can be proved analogously to Theorem \ref{theo1}, using that for $x \in {\mathbb R}$ we have $\cos ( {\mathrm i} \, x) = \cosh x$ and $\sin({\mathrm i} \, x) = {\mathrm i} \, \sinh x$, where $\sinh$ is an odd function with only one zero $x=0$.
Moreover, we can generalize Theorem \ref{theo2}.
\begin{corollary}
Let $\phi(t) = \gamma \cos(2\pi {a} + {b}) = \gamma \cosh((-{\mathrm i} (2\pi {a} t + {b}))$ for $(a, \, b) \in {\mathrm i} (0, \, \infty) \times {\mathrm i} {\mathbb R}$, 
and $P>0$. Then $\phi$ possesses the Fourier series $\phi(t) = \sum\limits_{n \in {\mathbb Z}} c_{n}(\phi) \, {\mathrm e}^{2\pi {\mathrm i} nt/P}$ with
\begin{eqnarray*} \mathrm{Re} \, c_n(\phi)  &=&   \frac{\gamma |a| P}{\pi (P^2 |a|^2 + n^2)} \sinh(\pi |a| P) \,  \cosh((-{\mathrm i} (\pi a P + b)), \\
 \mathrm{Im} \, c_n(\phi) &=&  \frac{ \gamma n}{\pi(|a|^2 P^2 + n^2)} \sinh(\pi |a|  P) \, \sinh(-{\mathrm i}(\pi a  P + b)).
\end{eqnarray*}
In particular, the Fourier coefficients $c_{n}(\phi)$ do not vanish  for all $n \in {\mathbb N}$. 
\end{corollary}

Thus, the Fourier coefficients of the signal in model (\ref{1.1a}) have still the same structure as found for the model (\ref{1.1}) in Section \ref{sec:rat}.
More precisely, for $f_{j}(t) = \gamma_{j}\cos(2\pi a_{j} t + b_{j})$ with $(a_{j}, \, b_{j}) \in {\mathrm i} (0, \, \infty) \times {\mathrm i} {\mathbb R}$, we also have 
$$ c_{n}(f_{j}) = \frac{A_{j} + {\mathrm i} B_{j} n}{n^{2}-C_{j}},$$
with 
\begin{eqnarray*} C_{j} &=& -|a_{j}|^{2} P^{2} = a_{j}^{2} P^{2}, \\
A_{j} &=& \frac{\gamma_{j} |a_{j}| P}{\pi} \sinh(\pi |a_{j}| P) \cosh((-{\mathrm i}(\pi a_{j} P + b_{j}|)) = 
-\frac{\gamma_{j} a_{j} P}{\pi} \sin(\pi a_{j} P) \cos(\pi a_{j} P + b_{j}), \\
B_{j} &=& \frac{\gamma_{j}}{\pi} \sinh( \pi |a_{j}| P) \, \sinh(-{\mathrm i} (\pi a_{j} P + b_{j})) = -\frac{\gamma_{j}}{\pi} \sin( \pi a_{j} P) \, \sin(\pi a_{j} P + b_{j}). 
\end{eqnarray*}
Hence, the obtained parameters $A_{j}$, $B_{j}$, $C_{j}$ have exactly the same form as in (\ref{cj})--(\ref{bj}). Moreover, we can reconstruct the parameters $\gamma_{j}, \, a_{j}, \, b_{j}$ of $f$ in (\ref{1.1a}) from $A_{j}$, $B_{j}$, $C_{j}$ via a generalization of Theorem \ref{theo3}.

\begin{corollary}\label{coro3}
Let $f(t)$ be given as in $(\ref{1.1a})$ with $\gamma_{j} \in (0, \infty)$ and 
with either $(a_j, \, b_{j})  \in  (0, \infty) \times [0, \, 2\pi)$ with $a_{j} \not\in \frac{1}{P}{\mathbb N}$ or 
$(a_j, \, b_{j})  \in {\mathrm i} (0, \infty) \times {\mathrm i} {\mathbb R}$. 
Then, there is a bijection between the parameters $\gamma_j, \, a_j, \, b_j$, $j=1, \ldots , K$, determining $f(t)$ 
and the parameters $A_j,\,  B_j, \, C_j$ in $(\ref{cj})-(\ref{bj})$, for $j=1, \ldots , K$.  
For $C_{j} >0$ we obtain $a_{j}, \, \gamma_{j}, \, b_{j}$ via Theorem $\ref{theo3}$.
For $C_{j} < 0$, we find
\begin{eqnarray*}
 a_j &=& \frac{1}{P} {\mathrm i} \sqrt{|C_{j}|}, \qquad 
 \gamma_j = \frac{\pi}{\sqrt{|C_j|} \, \sinh(\sqrt{|C_{j}|} \pi)} \sqrt{A_j^2+C_{j} B_j^2},\\
 b_j &= & {\mathrm i} \left(  -(\mathrm{sign}\, B_{j}) \, \mathrm{arccosh} \, \left( \frac{A_{j}}{\sqrt{A_{j}^{2} + C_{j} B_{j}^{2}}} \right)  - \sqrt{|C_{j}|} \pi \right).
 \end{eqnarray*}
\end{corollary}
\begin{proof}
For $C_{j} < 0$ it follows that  $a_j= \textrm{i} \frac{\sqrt{|C_j|}}{P}$.  Further, we obtain  
$$ \gamma_{j}^{2} = \frac{\pi^{2}}{|C_{j}| \left(\sinh(\sqrt{|C_{j}|} \pi)\right)^{2}} (A_{j}^{2} + C_{j} B_{j}^{2}). $$
The parameter $\gamma_{j}>0$ is thus uniquely defined, since we always have $A_{j}^{2} + C_{j} B_{j}^{2} >0$. Finally, inserting the found parameters $\gamma_{j}$ and $a_{j}$ into  (\ref{aj}) and (\ref{bj}), we obtain  
$$ |\sqrt{|C_{j}|}\pi - {\mathrm i} b_{j}| =  \textrm{arccosh} \, \left( \frac{A_{j}}{\sqrt{A_{j}^{2} + C_{j} B_{j}^{2}}} \right) $$ 
and 
$ \textrm{sign} (\sqrt{|C_{j}|} \pi - {\mathrm i} b_{j}) =  \textrm{sign} \,  B_{j}$, where $\mathrm{arccosh}$  is the inverse of $\cosh$  and maps onto $[0,\infty)$. Note that for $C_{j} <0$, we  necessarily have $A_{j}  >0$.
\end{proof}

Therefore, Algorithm \ref{alg1} can   also be  applied to a set of Fourier coefficients of the generalized model (\ref{1.1a}) to obtain a rational function that approximates the Fourier coefficients of the non-periodic part of $f$.
Then, we apply Algorithm \ref{alg2} as before to find the partial fraction decomposition of the rational function as described in Section \ref{sec:partial}, and  can reconstruct the wanted parameters for the nonperiodic part of $f$  in (\ref{1.1a}) using  Corollary \ref{coro3}. Finally, if $f$ contains a periodic part $\phi_{2}$ as studied in Section \ref{sec:periodic}, i.e., if there are parameters $a_{j} \in \frac{1}{P} {\mathbb N}$, then $\phi_{2}$ can be reconstructed via Theorem \ref{theoperiodic}.

\begin{remark}
The model (\ref{1.1a})  for real non-periodic functions $f$ is the most general model, such that Fourier coefficients of $f$ can be written as in (\ref{cnrat}). 
In particular, complex values for $C_{j}$ cannot occur in (\ref{cnrat}) for real functions, since we always have $c_{-n}(f) = \overline{c_{n}(f)} $ for $n \in {\mathbb N}$.
\end{remark}

%%%%%%%%%%%%%%%%%%%%%%%%%%%%%%%%%%%%%%%%%%%%%%%%%%%%%%%%%%%%%%%%%%%%%%%%%%%%%%%%%%%%%%%%%%%%%%%%%%%%%%%%

\section{Numerical Experiments}

In this section we present some numerical experiments, which show that the considered reconstruction scheme provides very good reconstruction results  even for small frequency gaps, if $P$ is chosen suitably.
In the first example, we start with the signal from \cite{CMW16},
\begin{eqnarray}
\nonumber
 f(t) &=& \cos(2\pi (5 t)) + \cos(2\pi (4.9t)) + 2 \cos(2\pi t) + \cos(2\pi(0.96t)) \\
 \label{bsp1}
  & & + \cos(2\pi(0.92t)) +\cos(2\pi(0.9t)). 
\end{eqnarray}  
According to our model (\ref{1.1}), $f(t)$ is given by the parameter vectors 
$$ {\mathbf a} = (5.0, \, 4.9, \, 1.0, 0.96, 0.92, 0.9), \quad  {\mathbf b}= (0, \, 0, \, 0, \, 0, \, 0, \, 0),  \quad \gamra=(1, \, 1, \, 2, \, 1, \, 1, \, 1).$$
In \cite{CMW16}, this function has been considered in the interval $[0,20)$. But in this interval already $4$ of the $6$ terms are periodic.
We consider 
$f(t)$ first in the interval  $[0,4)$, i.e., we take $P=4$, see Figure \ref{fig:bsp1} (left).
Then, the signal has a periodic part $\phi_{2}(t) = \cos(2\pi (5 t)) +  2 \cos(2\pi t)$, while 
$\phi_{1}(t) = \cos(2\pi (4.9t))  + \cos(2\pi(0.96t)) + \cos(2\pi(0.92t)) +\cos(2\pi(0.9t))$ is non-$P$-periodic.
We want to reconstruct  $f(t)$ using the  Fourier coefficients $c_{n}(f)$ for $n=1, \ldots , 20$.
We apply Algorithm \ref{alg1} with $tol=10^{-13}$. \\
Algorithm \ref{alg1} starts with the initialization values of largest magnitude $c_{4}(f)$, $c_{20}(f)$. Then the algorithm takes the further interpolation points $c_{3}(f)$ , $c_{19}(f)$, $c_{5}(f)$, $c_{18}(f)$, $c_{1}(f)$ (in this order) before it stops after $6$ iteration steps with error $3.9 \cdot 10^{-17}$.
The first two terms of ${\mathbf w} \in {\mathbb C}^{7}$ vanish, indicating that $c_{4}(f)$ and $c_{20}(f)$ are not interpolated by the obtained rational function $p(t)$. Indeed, for $a_{1}=5$ and $a_{3}=1$, we have that $a_{1}P = 20$ and $a_{3}P = 4$ are integers and therefore $c_{20}(f)$ and $c_{4}(f)$ contain information about the periodic part of $f(t)$. After omitting these two terms in ${\mathbf w}$ and in the corresponding index vector ${\mathbf S}$, we get $r(z)$ of order $(3,4)$ of the form 
(\ref{rat0}) with 
$$ 
{\mathbf S} = \left( \begin{array}{r}
3\\
19 \\
5\\
18\\
1 \end{array} \right) , \qquad 
{\mathbf w} = \left( \begin{array}{r}
0.0015811774 +  0.0002213439 \, {\mathrm i} \,\\
0.3562788014 -  0.0498743179 \, {\mathrm i} \,\\
0.0076223342 -  0.0010670670 \, {\mathrm i}\\
-0.9237779303 +  0.1293166867 \, {\mathrm i}\\
-0.0204814528 +  0.0028671323 \, {\mathrm i} \end{array} \right) 
. $$
Having found $r(z)$, the parameters of the non-$P$-peridic part $\phi_{1}(t)$ are  reconstructed  from $r(z)$ via Algorithm \ref{alg2} and Theorem \ref{theo3}.
The periodic part $\phi_{2}(t)$ of $f(t)$ is now determined according to Theorem \ref{theoperiodic} using (\ref{phi2c}).
All parameters can be recovered with high precision, where 
$$ \|{\mathbf a} - \tilde{\mathbf a} \|_{\infty} = 5.3 \cdot 10^{-11}, \qquad 
\|{\mathbf b} - \tilde{\mathbf b} \|_{\infty}= 8.5 \cdot 10^{-14}, \qquad 
\|{\gamra} - \tilde{\gamra} \|_{\infty}= 0.44 \cdot 10^{-9},  $$
where $\tilde{\mathbf a}$, $\tilde{\mathbf b}$ and $\tilde{\gamra}$ are the reconstructed  parameter vectors.

Taking the same setting with $L=40$ Fourier coefficients $c_{n}(f)$, $n=1, \ldots , 40$, the algorithm chooses the interpolation values 
$c_{4}(f)$, $c_{20}(f)$ for initialization,  and then $c_{3}(f)$, $c_{22}(f)$, $c_{5}(f)$, $c_{19}(f)$, $c_{1}(f)$ in this order before terminating with error $1.11 \cdot 10^{-16}$. In this case the parameters are reconstructed with errors 
$$ \|{\mathbf a} - \tilde{\mathbf a} \|_{\infty} = 1.7 \cdot 10^{-11}, \qquad 
\|{\mathbf b} - \tilde{\mathbf b} \|_{\infty}= 1.8 \cdot 10^{-10}, \qquad 
\|{\gamra} - \tilde{\gamra} \|_{\infty}= 0.27 \cdot 10^{-10}.  $$
We consider the same example for period $P=8$ and for given Fourier coefficients $c_{n}(f)$, $n=1, \ldots , 40$, see Figure \ref{fig:bsp1} (right).
In this case the algorithm starts with the initial values $c_{8}(f)$, $c_{7}(f)$ and then takes iteratively the interpolation values $c_{9}(f)$, $c_{40}(f)$, $c_{39}(f)$, $c_{38}(f)$ and $c_{6}(f)$, before terminating with error $1.8 \cdot 10^{-16}$.
The first and the 4th component of the vector ${\mathbf w} \in {\mathbb C}^{7}$ vanish and are removed. 
These components  are related  to the periodic part $\phi_{2}(t)$, since $a_{1}P = 40$ and $a_{3}P = 8$.
We obtain a rational function $r(z)$ of type $(3,4)$ given via (\ref{rat0}) with
$$ {\mathbf S} = \left( \begin{array}{r}
7\\
9 \\
39\\
38\\
6 \end{array} \right),
\qquad 
{\mathbf w} = \left( \begin{array}{r}
-0.0001586663 +  0.0001359792 \, {\mathrm i} \,\\
-0.0055153096 +  0.0047266956 \, {\mathrm i} \,\\
-0.1263871654 +  0.1083155261 \, {\mathrm i}\\
0.7486752234  -  0.6416248870 \, {\mathrm i}\\
0.0050447491 -   0.0043234189 \, {\mathrm i} \end{array} \right)
 . $$
The rational function $r(z)$ determines $\phi_{1}(t)$. Afterwards, $\phi_{2}(t)$ is reconstructed by Theorem \ref{theoperiodic} and  (\ref{phi2c}).
The parameter vectors are recovered by the algorithm with errors
$$ \|{\mathbf a} - \tilde{\mathbf a} \|_{\infty} = 9.9 \cdot 10^{-13}, \qquad 
\|{\mathbf b} - \tilde{\mathbf b} \|_{\infty}= 5.3 \cdot 10^{-14}, \qquad 
\|{\gamra} - \tilde{\gamra} \|_{\infty}= 4.3 \cdot 10^{-11}.  $$

\begin{figure}[htbp]
\centerline{\includegraphics[scale=0.45]{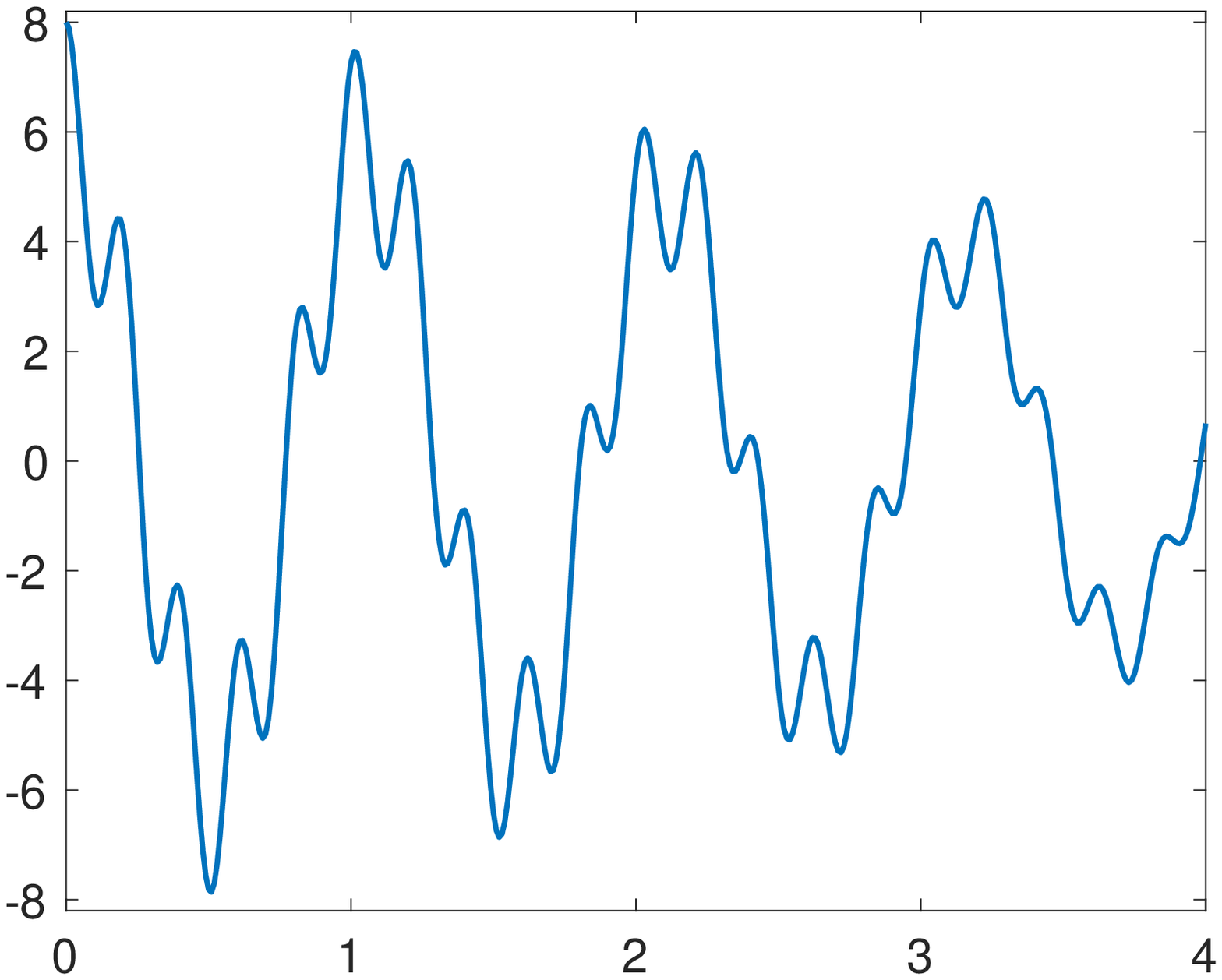} \quad \includegraphics[scale=0.45]{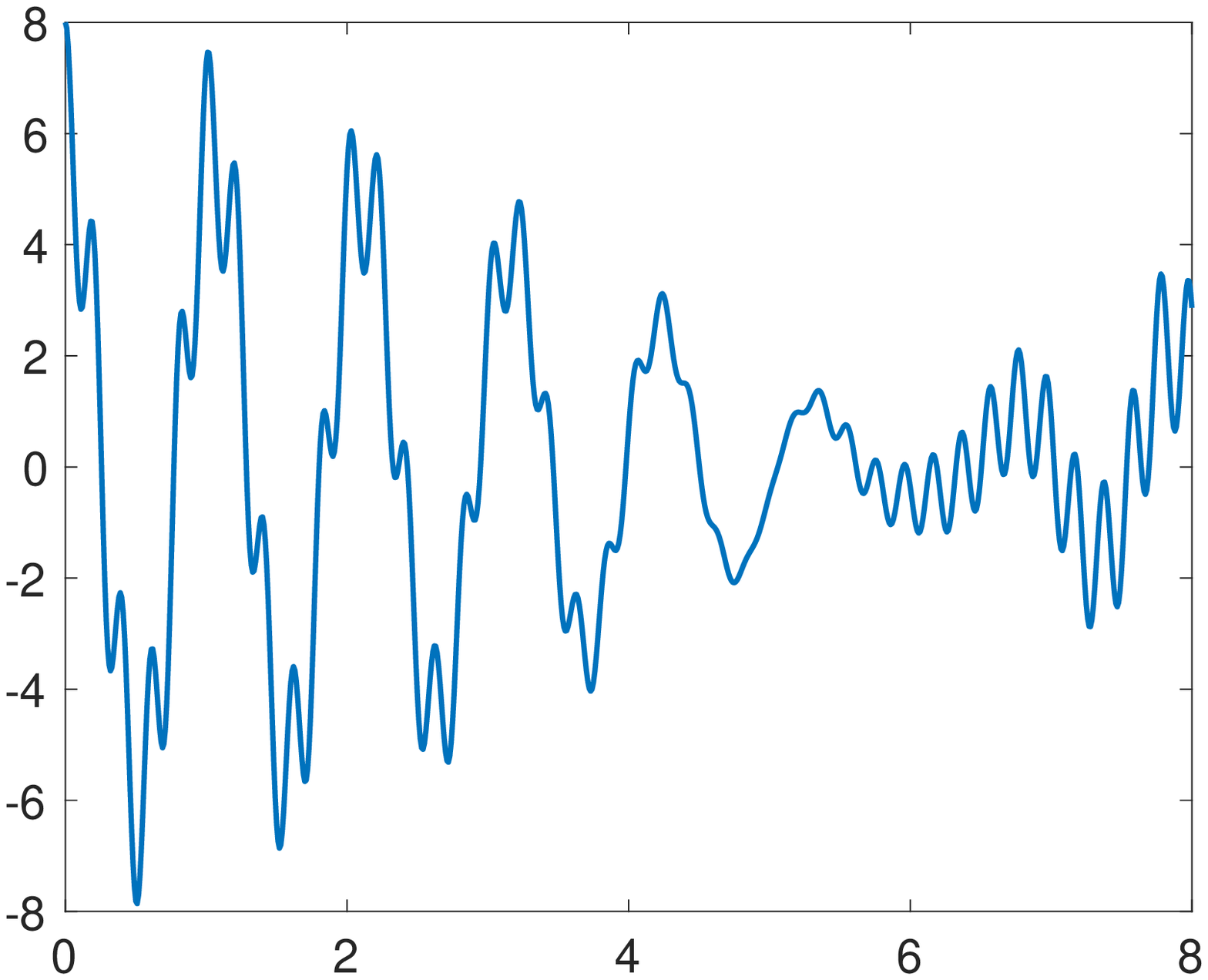}}
\medskip

\caption{Left: Plot of $f(t)$ in (\ref{bsp1}) on $[0,4]$. Right: Plot of $f(t)$ in (\ref{bsp1}) on $[0,8]$.\label{fig:bsp1}}
\end{figure}

In a second example we consider the function 
\begin{eqnarray}
\nonumber
f(t) &=&  0.5 \, \cos(2\pi (\sqrt{89}t) +0.5) +3 \cos(2\pi (\sqrt{29}t) +0.7) + 2 \cos(2\pi (\sqrt{21}t))
\\
\label{bsp2}
& & + 2\cos(2\pi (\sqrt{3}t)+0.3) +  \cos(2\pi \sqrt{2} +0.2) + \cos(2\pi(4t)+0.2),
\end{eqnarray}
i.e., $f(t)$ is given by the parameter vectors 
$${\mathbf a} = (\sqrt{89}, \, \sqrt{29}, \, \sqrt{21},\, \sqrt{3}, \, \sqrt{2}, 4), \;  {\mathbf b}= (0.5, \, 0.7, \, 0, \, 0.3, \, 0.2, \, 0.2)^{T}, \;
\gamra=(0.5, \, 3, \, 2, \, 2, \, 1,\, 1). $$
Taking $P=1$, this function has a periodic part $\phi_{2}(t)=\cos(2\pi(4t)+0.2)$, while  $\phi_{1}(t)$ consists of the other  five non-$1$-periodic terms, see Figure \ref{fig:bsp2}.
We employ 40 Fourier coefficients $c_{n}(f)$, $n=1, \ldots , 40$, for the recovery of $f$.
 Algorithm \ref{alg1} finds the values $c_{2}(f)$ and $c_{1}(f)$ for initialization. At the next iterations steps the values $c_{6}(f)$, $c_{5}(f)$, $c_{9}(f)$, $c_{10}(f)$, $c_{40}(f)$, $c_{4}(f)$, are taken for interpolation before the algorithm terminates with error $5.59 \cdot 10^{-17}$ after $7$ iteration steps.
The last component of ${\mathbf w} \in {\mathbb C}^{8}$ (which is related to the periodic part of $f$ since $a_{4} P = 4$) is vanishing and will be skipped.
We get a rational function $r(z)$ of type $(5,6)$, determined by (\ref{rat0}) via
$$ {\mathbf S} = \left( \begin{array}{r} 
2 \\ 1 \\ 6 \\ 5 \\ 9 \\ 10 \\ 40 \end{array} \right), \qquad 
{\mathbf w} = \left( \begin{array}{r}
 0.0002182770 - 0.0000242274\, {\mathrm i}\\
-0.0002310217 + 0.0000247469\, {\mathrm i}\\
 0.0045449823 - 0.0008522756\, {\mathrm i}\\
 0.0007235011 - 0.0001089651\, {\mathrm i}\\
 0.0013211292 + 0.0024068893\, {\mathrm i}\\
 0.0101627677 + 0.0030554091\, {\mathrm i}\\
-0.9998967232 + 0.0080228597\, {\mathrm i} 
 \end{array} \right). $$
 One Froissart doublet occurs in $r(z)$. This is due to the fact that the Fourier coefficient $c_{4}$ corresponding to the periodic part of $f$ has  been chosen for interpolation  by Algorithm \ref{alg1} only in the last iteration step.  According to the proof of Theorem \ref{theoperiodic}, we therefore need $7$ iteration steps  to generate a kernel of ${\mathbf A}_{8}$ of  dimension $2$. Application of Algorithm \ref{alg2} then leads to $6$  finite  eigenvalues $C_{1}, \ldots ,C_{6}$ of (\ref{eig}), while the equation system at the second step of Algorithm \ref{alg2} yields a vector $(A_{j} + {\mathrm i} B_{j})_{j=1}^{6}$ with one vanishing component. This component and the corresponding component $C_{j}$ are removed to obtain the rational function of type $(4,5)$ determining the non-periodic part $\phi_{2}$ of $f $. 
We reconstruct the  parameter vectors $\tilde{\mathbf a}$, $\tilde{\mathbf b}$ and $\tilde\gamra$ according to Theorem \ref{theo3} and Theorem \ref{theoperiodic} with (\ref{phi2c}).
with errors
$$ \|{\mathbf a} - \tilde{\mathbf a} \|_{\infty} = 9.6 \cdot 10^{-13}, \qquad 
\|{\mathbf b} - \tilde{\mathbf b} \|_{\infty}= 2.7 \cdot 10^{-12}, \qquad 
\|{\gamra} - \tilde{\gamra} \|_{\infty}= 3.4 \cdot 10^{-12}.  $$

\begin{figure}[htbp]
\centerline{\includegraphics[scale=0.45]{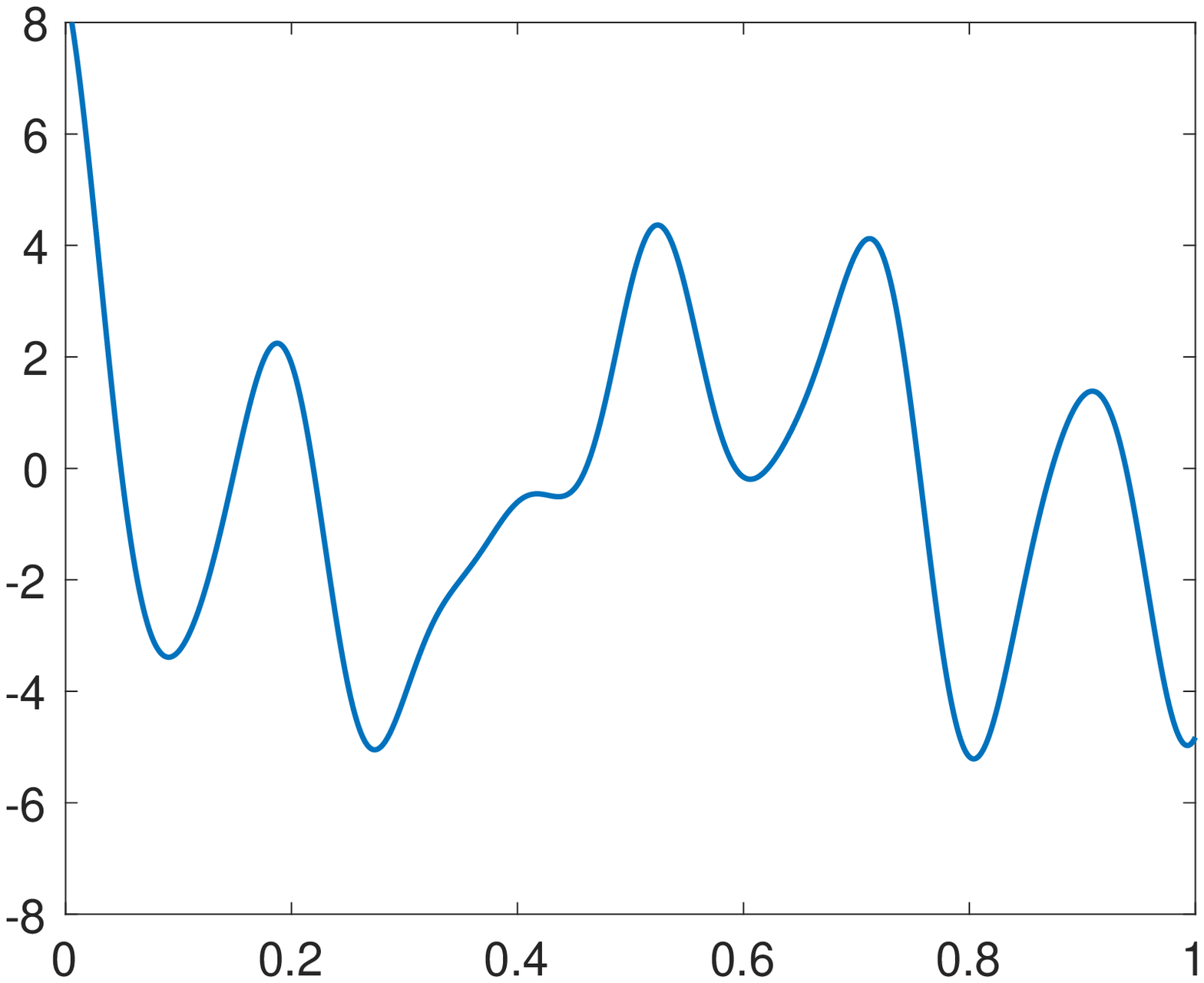}\quad \includegraphics[scale=0.45]{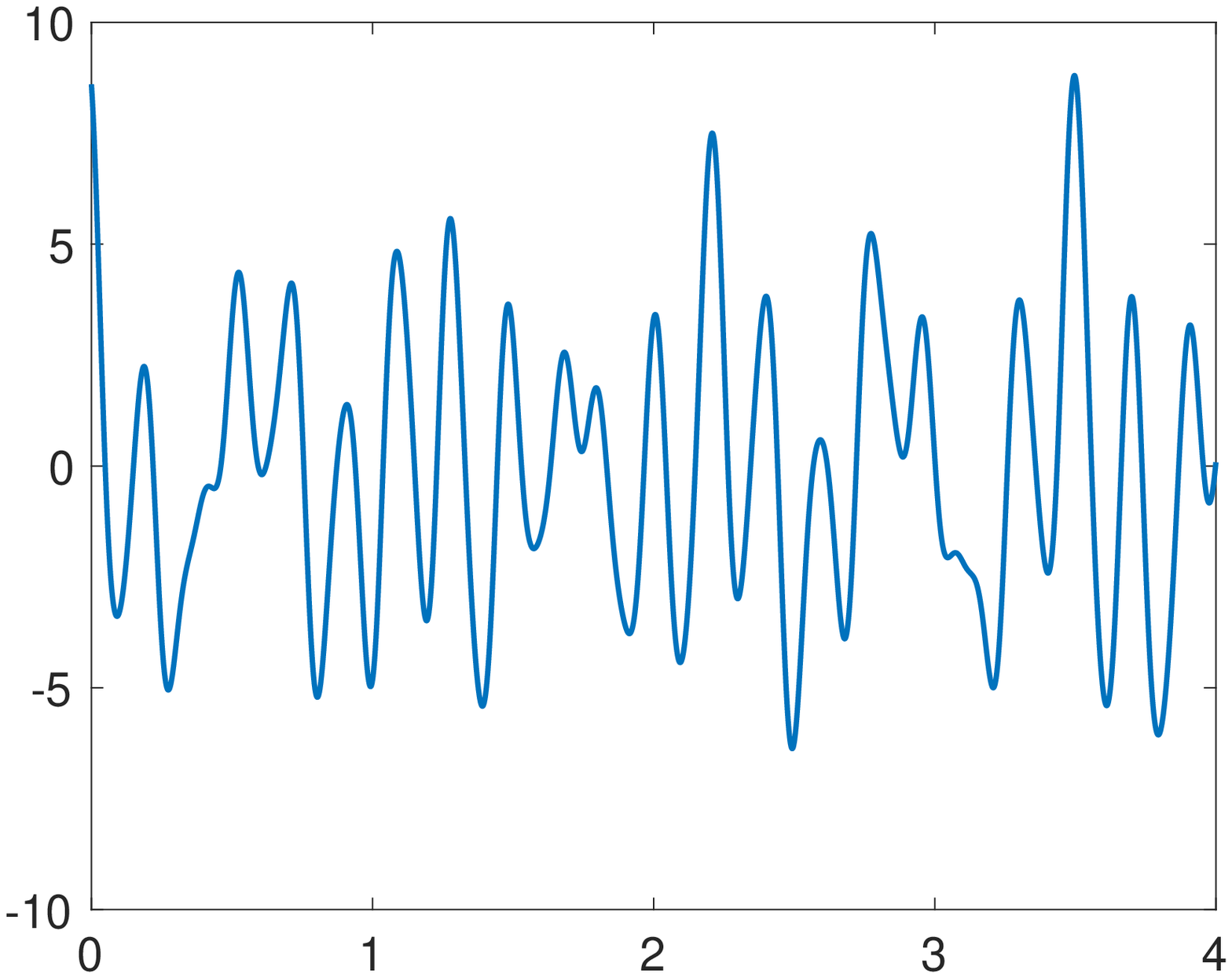}}
\caption{Plot of $f(t)$ in (\ref{bsp2}) on $[0,1]$ and on $[0,4]$.\label{fig:bsp2}}
\end{figure} 

\subsection*{Acknowledgement}
The authors gratefully acknowledge support by the German Research Foundation in the framework of the RTG 2088.

%%%%%%%%%%%%%%%%%%%%%%%%%%%%%%%%%%%%%%%%%%%%%%%%%%%%%%%%%%%%
\small
\bibliographystyle{plain}
\bibliography{bibliography.bib}

\begin{thebibliography}{10}

\bibitem{BDB10}
J.~Berent, P.L. Dragotti, and T.~Blu.
\newblock Sampling piecewise sinusoidal signals with finite rate of innovation
  methods.
\newblock {\em IEEE Trans. Signal Process.}, 58(2):613--625, 2010.

\bibitem{Berg86}
L.~Berg.
\newblock {\em Lineare {G}leichungssysteme mit {B}andstruktur und ihr
  asymptotisches {V}erhalten}.
\newblock Deutscher Verlag der Wissenschaften, Berlin, 1986.

\bibitem{BM05}
G.~Beylkin and L.~Monz\'{o}n.
\newblock On approximation of functions by exponential sums.
\newblock {\em Appl. Comput. Harmon. Anal.}, 19:17--48, 2005.

\bibitem{CMW16}
C.K. Chui, H.N. Mhaskar, and M.D. van~der Walt.
\newblock Data-driven atomic decomposition via frequency extraction of
  intrinsic mode functions.
\newblock {\em Int. J. Geomath.}, 7:117--146, 2016.

\bibitem{CL18}
A.~Cuyt and W.-s. Lee.
\newblock How to get high resolution results from sparse and coarsely sampled
  data.
\newblock {\em Appl. Comput. Harmon. Anal.}, 48(3):1066--1087, 2020.

\bibitem{DLW11}
I.~Daubechies, J.~Lu, and H.-T. Wu.
\newblock Synchrosqueezed wavelet transforms: An empirical mode
  decomposition-like tool.
\newblock {\em Appl. Comput. Harmon. Anal.}, 30(2):243--261, 2011.

\bibitem{FMP12}
F.~Filbir, H.N. Mhaskar, and J.~Prestin.
\newblock On the problem of parameter estimation in exponential sums.
\newblock {\em Constr. Approx.}, 35(3):323--343, 2012.

\bibitem{FNTB18}
S.-I. Filip, Y.~Nakatsukasa, L.N. Trefethen, and B.~Beckermann.
\newblock Rational minimax approximation via adaptive barycentric
  representations.
\newblock {\em SIAM J. Sci. Comput.}, 40(4):A2427--A2455, 2018.

\bibitem{FH07}
M.S. Floater and K.~Hormann.
\newblock Barycentric rational interpolation with no poles and high rates of
  approximation.
\newblock {\em Numer. Math.}, 107:315--331, 2007.

\bibitem{Huang}
N.E. Huang, Z.~Shen, S.R. Long, M.C. Wu, H.H. Shih, Q.~Zheng, N.-C. Yen, C.C.
  Tung, and H.H. Liu.
\newblock The empirical mode decomposition and the {H}ilbert spectrum for
  nonlinear and non-stationary time series analysis.
\newblock {\em Proc. Roy. Soc. A}, 454:903--995, 1998.

\bibitem{Kad64}
M.I. Kadets.
\newblock The exact value of the {P}aley–{W}iener constant.
\newblock {\em Dokl. Akad. Nauk SSSR}, 155(6):1253--1254, 1964.

\bibitem{Klein}
G.~Klein.
\newblock {\em Applications of Linear Barycentric Rational Interpolation}.
\newblock PhD thesis Fribourg, Switzerland, 2012.

\bibitem{Lev40}
N.~Levinson.
\newblock {\em Gap and Density Theorems}.
\newblock Colloquium publications. American Mathematical Society, Providence,
  RI, 1940.

\bibitem{NST18}
Y.~Nakatsukasa, O.~Sete, and L.N. Trefethen.
\newblock The {AAA} algorithm for rational approximation.
\newblock {\em SIAM J. Sci. Comput.}, 40(3):A1494--A1522, 2018.

\bibitem{PPT11}
T.~Peter, D.~Potts, and M.~Tasche.
\newblock Nonlinear approximation by sums of exponentials and translates.
\newblock {\em SIAM J. Sci. Comput.}, 33(4):1920--1947, 2011.

\bibitem{PP19}
G.~Plonka and V.~Pototskaia.
\newblock Computation of adaptive {F}ourier series by sparse approximation of
  exponential sums.
\newblock {\em J. Fourier Anal. Appl.}, 25(4):1580--1608, 2019.

\bibitem{PPST18}
G.~Plonka, D.~Potts, G.~Steidl, and M.~Tasche.
\newblock {\em Numerical {F}ourier {A}nalysis}.
\newblock Birkh\"auser, Basel, 2018.

\bibitem{PSK18}
G.~Plonka, K.~Stampfer, and I.~Keller.
\newblock Reconstruction of stationary and non-stationary signals by the
  generalized {P}rony method.
\newblock {\em Anal. and Appl.}, 17(2):179--210, 2019.

\bibitem{PT14}
G.~Plonka and M.~Tasche.
\newblock Prony methods for recovery of structured functions.
\newblock {\em GAMM Mitt.}, 37(2):239--258, 2014.

\bibitem{PT10}
D.~Potts and M.~Tasche.
\newblock Parameter estimation for exponential sums by approximate {P}rony
  method.
\newblock {\em Signal Process.}, 90(5):1631--1642, 2010.

\bibitem{PT2013}
D.~Potts and M.~Tasche.
\newblock Parameter estimation for nonincreasing exponential sums by
  {P}rony-like methods.
\newblock {\em Linear Algebra Appl.}, 439(4):1024--1039, 2013.

\bibitem{QW11}
T.~Qian and Y.-B. Wang.
\newblock Adaptive {F}ourier series - a variation of a greedy algorithm.
\newblock {\em Adv. Comput. Math.}, 34:279--293, 2011.

\bibitem{RK89}
R.~Roy and T.~Kailath.
\newblock {ESPRIT} estimation of signal parameters via rotational invariance
  techniques.
\newblock {\em IEEE Trans. Acoust. Speech Signal Process.}, 37:984--995, 1989.

\bibitem{SW86}
C.~Schneider and W.~Werner.
\newblock Some new aspects of rational interpolation.
\newblock {\em Math. Comp.}, 47(175):285--299, 1986.

\bibitem{VMB02}
M.~Vetterli, P.~Marziliano, and T.~Blu.
\newblock Sampling signals with finite rate of innovation.
\newblock {\em IEEE Trans. Signal Process.}, 50(6):1417--1428, 2002.

\bibitem{Young80}
R.M. Young.
\newblock {\em An Introduction to Nonharmonic {F}ourier Series}.
\newblock Academic Press, New York, 1980.

\bibitem{ZP18}
R.~Zhang and G.~Plonka.
\newblock Optimal approximation with exponential sums by a maximum likelihood
  modification of {P}rony’s method.
\newblock {\em Adv. Comput. Math.}, 45(3):1657--1687, 2019.

\end{thebibliography}

%\bibliographystyle{siamplain}
%\bibliography{myBib.bib}

\end{document}